%% file: main.tex
\documentclass{article}
\usepackage[margin = 3.5cm]{geometry}

\usepackage{natbib}

\usepackage{authblk}
\usepackage{tikz}
\usepackage{pgfplots}
\usetikzlibrary{patterns}

\usepackage{tabularx}
\def\arraystretch{1}

\usepackage{comment}
\usepackage{graphicx}
\usepackage{commands, polyhedralCommands, amsmath, amssymb, amsthm}
\usepackage{appendix}
\usepackage{formZcommands}

\title{Dominance-based linear formulation for the Anchor-Robust Project Scheduling Problem}

\author[1,2]{Pascale Bendotti}
\author[2]{Philippe Chr\'etienne}
\author[2]{Pierre Fouilhoux}
\author[1,2]{Ad\`ele Pass-Lanneau}

\affil[1]{ 
\normalsize
EDF R\&D, F-91120 Palaiseau, France\\
\hspace{5cm} \texttt{\{pascale.bendotti, adele.pass-lanneau\}@edf.fr}}

\affil[2]{ Sorbonne Universit\'e, CNRS, LIP6, F-75005 Paris, France\\
\hspace{5cm} \texttt{\{philippe.chretienne, pierre.fouilhoux\}@lip6.fr}}

\date{February 13, 2021}

\begin{document}

\maketitle

\begin{abstract}
\input{formZabstract.tex}

\emph{Keywords:} project scheduling; combinatorial optimization; mixed-integer programming; robust 2-stage optimization; polyhedral characterization.
\end{abstract}

\input{formZbody}

\bibliographystyle{plainnat}
\bibliography{anchor}

\end{document}

%% file: formZabstract.tex
In project scheduling under processing times uncertainty, the Anchor-Robust Project Scheduling Problem is to find a baseline schedule of bounded makespan and a max-weight  subset of jobs whose starting times are guaranteed.
The problem was proven NP-hard even for budgeted uncertainty.
In the present work we design mixed-integer programming (MIP) formulations 
that are valid for a variety of uncertainty sets encompassing budgeted uncertainty.
A new dominance among solutions is proposed, resulting into an MIP formulation.
We further study the combinatorial structure of the problem. 
Non-trivial polynomial cases under budgeted uncertainty are exhibited, where the dominance-based formulation yields a polyhedral characterization of integer solutions. In more general cases, the dominance-based formulation is shown to be tighter than all previously known formulations.
In numerical experiments we investigate how the formulation performs on instances around the polynomial cases, for both budgeted uncertainty sets and more elaborate uncertainty sets involving several budgets.

%% file: formZbody.tex
\section{Introduction}

Consider a set of $n$ jobs $\Jobs$ to be scheduled under precedence constraints.
The \emph{precedence graph} is a directed acyclic graph $G(p)=(\AllJobs, \Arcset, p)$, where $\AllJobs = \Jobs \cup \{s,t\}$ and $s$ (resp. $t$) is a dummy job representing the beginning (resp. end) of the project.
Processing times of jobs are $p \in \RealJp$ and every arc $(i,j) \in \Arcset$ has length $p_i$, with $p_s=0$.
A schedule of $G(p)$ is a vector $x \in \Real_+^{\AllJobs}$ of starting times such that $x_s = 0$ and $x_j - x_i \geq p_i$ for every arc $(i,j)$ of $G(p)$.
Finding a schedule with minimum \emph{makespan} $x_t = \min_{i \in J} x_i + p_i$ is a classical polynomial-time solvable problem \citep{Pinedo}.

In practice processing times may be uncertain parameters, which calls for robust approaches. We consider that jobs have nominal processing times $p$ and the real processing times may be $p+\delta$, where $\delta$ lies in an uncertainty set $\Delta \subset \RealJp$.
For a given uncertainty set, different robust approaches can be considered.
A first approach from the \emph{static-robust} framework \citep{Soyster} is to find a schedule $x \in \Real_+^{\AllJobs}$ such that $x$ is a schedule of $G(p\!+\!\delta)$ for every $\dinD$,
and so that the makespan of $x$ is minimized.
The static-robust approach is known to be overly conservative, i.e., it produces a schedule with a very large makespan.
An alternate approach from \emph{adjustable-robust optimization} \citep{BGGNAdjustable}, 
is to find the minimum value $B$ such that for every $\dinD$ there exists a schedule  $x^{\delta} \in \Real_+^{\AllJobs}$ of $G(p\!+\!\delta)$ such that $\xd_t \leq B$.
The adjustable-robust approach is appealing since it often gives a worst-case makespan that is significantly lower than the static-robust makespan.
It was studied in \citet{MinouxEncyclo} for project scheduling. A major drawback is that the schedule depends on the uncertainty realization. In other words, starting times cannot be decided before $\delta$ is known.

Project scheduling has a large variety of practical applications such as planning industrial activities. The durations of activities are often not known exactly or they can change over time: this calls for the study of project scheduling under processing times uncertainty. It is often necessary to compute a schedule in advance, called \emph{baseline schedule} \citep{HerroelenLeusProjectSchedulingUncertainty}, while taking into account processing times uncertainty. The precomputation of a baseline schedule is common practice, especially when scheduling activities requires preparation and coordination with other entities to
secure the availability of staff or specific equipment. 
The makespan of the baseline schedule is a major criterion; in some cases, the project must be scheduled so that it completes before a deadline.
Moreover if processing times are disrupted and the schedule must be revised, it is important to stick to the baseline schedule \citep{HerroelenStablePreschedule}. As some activities may be difficult to reschedule, it is interesting to guarantee their starting times in the baseline schedule against uncertainty realizations.
\medbreak
\emph{The Anchor-Robust Project Scheduling Problem.} A new robust problem was introduced in \citet{AnchRobPSPHal}.
Before the real processing time of jobs is known, the problem is to decide which solution $x$ of $G(p)$ to choose as a baseline schedule with a bounded makespan. The key idea is to define an anchored set of jobs as a subset of jobs whose starting times could remain the same for any realization in the uncertainty set. 
Formally, given a schedule $x$ of $G(p)$ and uncertainty set $\Delta$, a subset of jobs $H \subseteq \Jobs$ is \emph{$x$-anchored} if for every $\dinD$ there exists $x^{\delta}$ schedule of $G(p\!+\!\delta)$ such that $x^{\delta}_i = x_i$ for every $i \in H$.
Given a deadline for the project and an anchoring weight associated to each job, the Anchor-Robust Project Scheduling Problem (AnchRob) is then to find a baseline schedule $x$ satisfying the deadline, and a subset of jobs $H$ that is $x$-anchored, so as to maximize the total anchoring weight of $H$.

Connections with classical robust approaches are as follows.
If there exists a solution $(x,H)$ of (AnchRob) with $H = J$, i.e., a solution where all jobs are anchored, then $x$ is a solution of the static-robust problem. If no such solution exists, it means that no static-robust solution respects deadline $M$. (AnchRob) is a so-called \emph{robust 2-stage} optimization problem, where  $x$ and $H$ are first-stage decisions and $x^{\delta}$ are second-stage decisions. Schedule $x$ is thus a baseline schedule which may be adjusted in second stage by revising starting times of non-anchored jobs only.

Let us present several uncertainty sets of interest for (AnchRob). This includes \emph{budgeted uncertainty sets} \citep{PriceBS} where at most $\Gamma$ processing times deviate from the nominal values $p$.
A special case is \emph{box uncertainty sets}, in which every processing time deviates from its nominal value within an interval, corresponding to uncertainty budget $\Gamma = n$. 
Another special case, referred to as \emph{1-disruption uncertainty sets}, is also considered. It corresponds to the situation where one deviation of fixed length occurs to any job in the project (i.e., one job, whichever it could be).
We also consider more elaborate uncertainty sets, corresponding to unions or intersections of budgeted uncertainty sets. Such uncertainty sets account, e.g., for the case of several uncertainty budgets on subsets of jobs.

\medbreak
Robust 2-stage problems are often considered challenging and computationally intensive \citep{BuchheimKurtzSurveyEJCO,surveyGabMurThi}.
A compact MIP formulation for robust 2-stage problems is often not known, and proposed exact approaches rely on decomposition methods, see, e.g., \citep{Costa,AyoubPossDecomposition,ZZ}.
Another trend is to consider approximations. With \emph{decision rules} introduced in \citet{BGGNAdjustable} it is assumed that second-stage variables have a fixed dependency (e.g., affine) on the uncertainty realization. It yields tractable approximations of the original robust 2-stage problem.
With \emph{k-adaptability} introduced by \citet{BertsimasCaramanisFiniteAdaptability} a fixed number of recourse solutions is determined in first stage and the recourse problem is to choose the best among them.
The focus of the present work is to investigate exact approaches for the (AnchRob) problem, and especially compact formulations, which -- to the best of our knowledge -- have been scarcely studied for robust 2-stage problems.

Robust 2-stage problems including a baseline solution in first stage and a feature to keep decisions unchanged were considered under the name of \emph{recoverable robustness} \citep{recoverableLiebchen,RecoverableTimetables} 
or \emph{robust optimization with incremental recourse} \citep{OrlinIncremental}.
When decisions are represented by continuous variables, different measures can be used to account for the stability of decisions. For project scheduling, the expected absolute gap between starting times was used as stability measure in \citet{HerroelenStablePreschedule}.
The anchoring criterion was defined as the number of identical starting times in \citet{CPMEJOR}. 
This criterion was studied for rescheduling problems in \citet{ISCOAnchResched}. 
The (AnchRob) problem was introduced in \citet{AnchRobPSPHal}, and it was proven NP-hard even for budgeted uncertainty. For budgeted uncertainty, a compact MIP reformulation was obtained, denoted by (Lay). It is based on a so-called layered graph and a dedicated analysis of the problem in this special case.
This MIP formulation is inherent to budgeted uncertainty, and thus not applicable to other uncertainty sets.

\medbreak
In this work, we investigate linear formulations for (AnchRob) that are valid for a variety of uncertainty sets encompassing budgeted uncertainty.
The starting point is to precompute the worst-case value over $\Delta$ of the longest \ijpath path for every pair of jobs $i,j$, and to use these values in a linear formulation.
We exhibit a dominance property among schedules and derive a linear formulation from it.
This dominance-based linear formulation, called (Dom), improves over a naive linear formulation.
In the case of budgeted uncertainty, (Dom) is compared with known formulation (Lay).
A polyhedral study is carried out to highlight how formulation (Dom) captures the combinatorial structure of the problem.
We prove that (Dom) yields a complete polyhedral characterization in two special cases: box uncertainty, and 1-disruption uncertainty for precedence graphs where all \stpath paths are longest \stpath paths.
We thus obtain that the latter is a polynomial case, while (AnchRob) under budgeted uncertainty is NP-hard in general.
Finally numerical experiments give evidence that (Dom) performs well for budgeted uncertainty, even for instances that do not match the polyhedral characterization cases. (Dom) is also capable of solving the problem for uncertainty sets where no MIP formulation was previously investigated, e.g., in the case of several budgets.
\medbreak

In Section~\ref{sec:prelim} notation and preliminaries on problem (AnchRob) are given.
In Section~\ref{sec:formulation} the dominance-based linear formulation is presented, and {compared with} other known formulations.
In Section~\ref{sec:polyhedralCharac} polyhedral characterizations are proven for box uncertainty, and 1-disruption uncertainty.
In Section~\ref{sec:num} numerical results are presented.

\section{Preliminaries}
\label{sec:prelim}

\subsection{General notation}

Let $G(p)=(\AllJobs, \Arcset, p)$ be the precedence graph.
It is assumed that there is an arc $(s,i)$ (resp. an arc $(i,t)$) for every job $i \in J$ without predecessor (resp. without successor) in $J$.
Let $\prec$ denote the partial order on $\AllJobs$ defined by $i \prec j$ if there exists an \ijpath path in $G$.
Given $i,j \in \AllJobs$, $i \prec j$, let $\LGpij$ be the length of the longest \ijpath path in $G(p)$. 
In the sequel we will use the shorthand notation $\Lzij = \LGpij$.
The minimum makespan of a schedule of $G(p)$ is then $\Lz_{st}$. 
The \emph{earliest schedule} defined by $x_i = \Lz_{si}$ for every $i \in \Jobs$ is a schedule with minimum makespan.
Let $\SetSchedules \subseteq \RealAllJp$ denote the set of schedules of $G(p)$ with makespan at most $M$.
A longest \stpath path is a \emph{critical path}. 
The precedence graph $G(p)$ is \emph{quasi-critical} if every job $i \in J$ belongs to a critical path, i.e., $\Lz_{si} + \Lz_{it} = \Lz_{st}$.
The precedence graph $G(p)$ is \emph{critical} if all \stpath 
paths in $G(p)$ are critical, i.e., they have length $\Lz_{st}$.
Note that if $G(p)$ is critical then it is quasi-critical.
The uncertainty set $\Delta$ is assumed to be a subset of $\RealJp$.
For uncertainty set $\Delta$, let us define the value $$\LDij = \max_{\dinD} \LGpdij$$ 
for every $i,j \in \AllJobs$, $i \prec j$. Then $\LDij$ is the \emph{worst-case longest path value} between job $i$ and job $j$.

\subsection{The Anchor-Robust Project Scheduling Problem}

Let us now give a formal definition of the Anchor-Robust Project Scheduling Problem (AnchRob). An instance of the problem is described by the parameters $(G(p),M,$ $\Delta,w)$ where: $G(p)=(\AllJobs, \Arcset, p)$ is the precedence graph, $M \geq 0$ is a deadline on the project, $\Delta\subset \RealJp$ is the uncertainty set, and $w \in \RealJp$ is a vector of non-negative anchoring weights associated with jobs.

A solution of the problem is a pair $(x,H)$ where $x$ is a schedule of $G(p)$ with makespan $x_t$ at most $M$, and $H$ is a subset of $J$ that is $x$-anchored, that is: for every realization $\delta$ in $\Delta$, there exists a schedule $\xdelta$ of $G(p+\delta)$ satisfying $\xdelta_i=x_i$ for $i\in H$.
The objective is to find a solution $(x,H)$ with maximum total anchoring weight $\sum_{i \in H} w_i$.

\medbreak

The Anchor-Robust Project Scheduling Problem is related to robust 2-stage optimization. First-stage decisions are the baseline schedule $x$, and the anchored set $H$. Second-stage decisions, i.e., decisions that depend on the uncertainty realization, are the schedule $\xdelta$ involved in the definition of an anchored set. The problem can be rewritten under the following form:\\
\begin{center}
\def\arraystretch{1}
\begin{tabularx}{\textwidth}{l r l X}
(AnchRob): &    $\max$ & $\sum_{i\in H} \aw_i$  \\
&     s.t. &  $x$ schedule of $G(p)$ \\ 
&          & $x_t \leq M$ \\

&          & $H \subseteq \Jobs$: $\forall \dinD$, $\exists\ x^{\delta}$ schedule of $G(p\!+\!\delta)$\\
&          & \hspace{2.7cm} s.t. $x_i = x^{\delta}_i$ $\forall i\!\in\!H$
& \hfill\\
\end{tabularx}
\end{center}
Note that the proposed form is related to max/min/max problems arising usually in robust 2-stage optimization.
Given a schedule $x$ of $G(p)$, a subset $H \subseteq \Jobs$, and an uncertainty realization $\dinD$, the second-stage or recourse problem of (AnchRob) is to decide the existence of a schedule $\xdelta$ of $G(p+\delta)$ with $x_i = \xdelta_i$ for every $i \in H$.
\medbreak
Let us finally illustrate (AnchRob) on a simple example.
Consider a project scheduling instance with 5 jobs. The precedence graph is represented in Figure~\ref{fig:precGraph}. Nominal processing times are $p=(1,1,1,1,2)$. 
Each job is also associated with a worst-case deviation $\dhat = (0.5, 1, 0.5, 0.5, 0.5)$. In Figure~\ref{fig:precGraph}, each arc $(i,j)$ is weighted with value $p_i+\dhat_i$.
Let the deadline be $M = 4.5$.

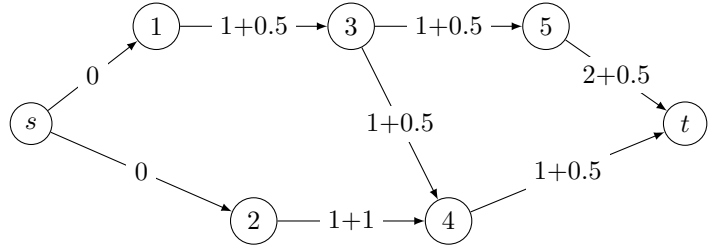
\begin{figure}[H]
    \centering
    \input{precGraph}
    \caption{Precedence graph for an instance with $5$ jobs, with arc-weights $p+\dhat$.}
    \label{fig:precGraph}
\end{figure}

Two different uncertainty sets will be considered, to give a flavor of their respective impact on solutions of (AnchRob).

A first uncertainty set is the box $\Delta = \Pi_{i \in J} [0, \dhat_i]$. Then every job $i$ may have any duration in range $[p_i,\ p_i\!+\!\dhat_i]$.
Consider a schedule for processing times $p+\dhat$. If it has a makespan larger than $M$, such a schedule is infeasible for (AnchRob).
Hence (AnchRob) is different from the static-robust case.
Consider the schedule $x = (0, 1, 1, 3, 2.5)$ and set $H = \{1,2,4\}$.
Figure~\ref{fig:scheduleBox} shows schedule $x$. Each job is represented by a rectangle of length $p_i$, and deviation $\dhat_i$ is represented with dotted rectangle. Jobs from $H$ are represented in dark gray.
 Then $(x,H)$ is a solution of (AnchRob) for box $\Delta$. First $x$ has makespan $M=4.5$. Also the set $H$ is $x$-anchored, since for any value of jobs durations, it is possible to repair the schedule by moving jobs 3 and 5 only (in this case, just by right-shifting).

\begin{figure}[H]
    \centering
    \input{scheduleBox}
    \caption{Schedule $x = (0, 1, 1, 3, 2.5)$ and set $H = \{1,2,4\}$ in dark gray, $x$-anchored for box uncertainty set $\Delta = \Pi_{i \in J} [0, \dhat_i]$.}
    \label{fig:scheduleBox}
\end{figure}
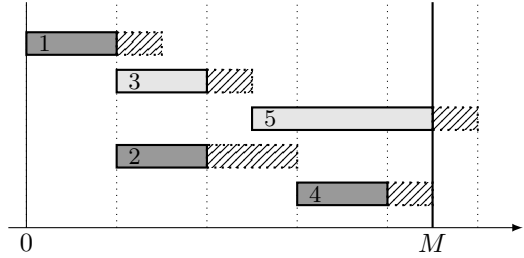

A second uncertainty set is 
 $\Delta = \{ (\dhat_i u_i)_{i \in \Jobs}\colon\ u \in \BinJobs,\ \sum_{i \in \Jobs} u_i \leq 1\}$. It corresponds to an uncertainty budget $\Gamma=1$, i.e., at most one processing time deviates. Consider the same schedule $x$, and the set $H' = \{1, 2, 4, 5\}$, represented in Figure~\ref{fig:scheduleBudgeted}. Set $H'$ is $x$-anchored for budgeted uncertainty set $\Delta$: indeed for any value of $\delta \in \Delta$, the schedule can be repaired by moving only job~3.

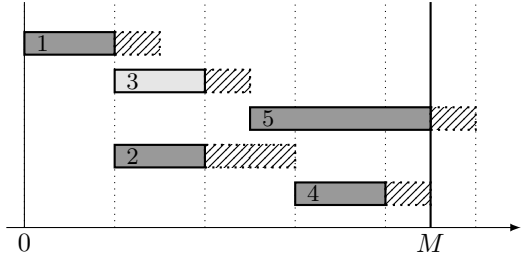
\begin{figure}[H]
    \centering
    \input{scheduleBudgeted}
    \caption{Schedule $x = (0, 1, 1, 3, 2.5)$ and set $H' = \{1,2,4,5\}$ in dark gray, $x$-anchored for budgeted uncertainty set $\Delta = \{ (\dhat_i u_i)_{i \in \Jobs}: u \in \BinJobs,\ \sum_{i \in \Jobs} u_i \leq 1\}$.}
    \label{fig:scheduleBudgeted}
\end{figure}

\subsection{Uncertainty sets}
\label{sec:uncSets}

Note first that some assumptions on $\Delta$ can be made without loss of generality.
 The uncertainty set can be assumed to be convex \citep{BGGNAdjustable}.
 The uncertainty set can also be assumed to be down-monotone, i.e., if $\delta \in \Delta$ and $\delta' \leq \delta$ then $\delta' \in \Delta$. Indeed it directly follows from the fact that if $x^{\delta}$ is a schedule of $G(p+\delta)$, then it is a schedule of $G(p+\delta')$ for every $\delta' \leq \delta$.
In the present work, considered uncertainty sets will be polyhedra, and w.l.o.g. down-monotone.

\medbreak
Let us now present some uncertainty sets of interest.

% BOX
Set $\Delta$ is a \emph{box uncertainty set} if $\Delta = \{ (\delta_i)_{i \in J}\ :\ 0 \leq \delta_i \leq \dhat_i\ \forall i \in J\}$ with $\dhat \in \RealJp$, i.e., it is a cartesian product of intervals. Then $\dhat$ is a greatest element of $\Delta$ in the sense that $\delta \leq \dhat$ for every $\dinD$. Note that if $\Delta$ is any set with greatest element $\dhat$, then w.l.o.g. it can be assumed to be down-monotone, and thus equal to the box with greatest element $\dhat$.

% BUDGET
Set $\Delta$ is a \emph{budgeted uncertainty set} if $\Delta = \{ (\dhat_i u_i)_{i \in \Jobs}: u \in \CubeJ,\ \sum_{i \in \Jobs} u_i \leq \Gamma\}$, with \emph{deviation} $\dhat \in \RealJp$ and \emph{uncertainty budget} $\Gamma \in \{1, \dots, |J| \}$. The extreme points of such $\Delta$ correspond to budgeted uncertainty as defined in the seminal work of \citep{PriceBS}, where $\Gamma$ is the number of processing times that may deviate from their nominal value.
Box uncertainty is the special case of budgeted uncertainty where $\Gamma = |\Jobs|$.

% 1-D
Set $\Delta$ is a \emph{1-disruption uncertainty set} if it is a budgeted uncertainty set with unit budget $\Gamma = 1$ and uniform deviation, i.e., $\dhat_i = \dhat_0$ for every $i \in J$. Extreme points of a 1-disruption uncertainty set represent the situation where one event of fixed -- possibly large -- deviation $\dhat_0$ may happen anywhere in the project.
Then the processing time of one job is increased by fixed amount $\dhat_0$.

\medbreak
Let us now present more elaborate uncertainty sets, built as unions or intersections of budgeted uncertainty sets.

Set $\Delta$ is a \emph{partition-budgeted uncertainty set} if $\Delta = \{ (\dhat_i u_i)_{i \in \Jobs}: u \in \CubeJ,$ $\sum_{i \in J^k} u_i \leq \Gamma^k\ \forall k \in \{ 1, \dots, m\}\}$ where $(J^1, ..., J^{m})$ is a partition of $J$ and $\Gamma^k \in \{1, \dots, |J^k|\}$ for every $k \in \{ 1, \dots, m\}$. Each group $J^k$ is associated with its own uncertainty budget $\Gamma^k$. It holds that $\Delta = \bigcap_{1 \leq k \leq m} \Delta^k$ where $\Delta^k = \{ (\dhat_i u_i)_{i \in \Jobs}: u \in \CubeJ,\ \sum_{i \in J^k} u_i \leq \Gamma^k \}$ for every $k \in \{1, \dots, m\}$. Distinct uncertainty budgets on disjoint subsets of the partition are relevant when deviations of jobs from different subsets are uncorrelated. A special case of interest is to consider a partition $(J^1, J^2)$ where jobs of the first subset $J^1$ 
are associated with small deviations but large uncertainty budget $\Gamma^1$; jobs of the second subset $J^2$ are associated with large deviations but small uncertainty budget $\Gamma^2$.

Set $\Delta$ is a \emph{mixed-budgeted uncertainty set} if $\Delta = \bigcup_{1 \leq k \leq m} \Delta^k $ where $\Delta^k$ are budgeted uncertainty sets. 
Consider the following special case, where $\Delta^1$ is defined by deviation $\dhat$ and budget $\Gamma^1$, and $\Delta^2$ is defined by deviation $\tau\dhat$ for a given $\tau \in [0, 1]$, and budget $\Gamma^2 > \Gamma^1$. Then $\Delta = \Delta^1 \cup \Delta^2$ supports two kinds of uncertainty realizations corresponding either to a large number of small deviations (i.e., $\delta \in \Delta^2$) or a small number of large deviations (i.e., $\delta \in \Delta^1$).
For uniform deviation, it holds that $\Delta^1 \subsetneq \Delta$ whenever $\tau \Gamma^2 > \Gamma^1$.

\medbreak
All these uncertainty sets are defined implicitly through inequalities or set operations.
Note that the uncertainty set may also be given explicitly as the convex hull of a discrete set of points.
Given a set $S$ of scenarii, for every $s \in S$ the uncertainty realization is some $\delta^s \in \RealJp$, and
$\Delta = \conv\{\delta^s, s \in S\}$.
Any implicitly-defined set $\Delta$ can be written under that form by enumerating its extreme points, but then the input size is exponentially increased.

\subsection{Assumption on longest paths computation}

Computing worst-case longest path values $\LDij$ can be an NP-hard problem for polyhedral set $\Delta$ defined by inequalities \citep{AnchRobPSPHal}.
 
However, efficient algorithms can be designed for some uncertainty sets.
For box uncertainty, the computation is straightforward since $\LDij = L_{G(p+\dhat)}(i,j)$.
For budgeted uncertainty sets, and thus 1-disruption uncertainty sets, the worst-case longest path values $\LDij$ can be computed in polynomial time by dynamic programming \citep{MinouxEncyclo}. The algorithm is linear in $\Gamma |\Arcset|$.
For partition-budgeted uncertainty sets, the $\LDij$ values can be computed using a straightforward generalization of the dynamic programming of \citep{MinouxEncyclo} which is linear in $(\Pi_{1 \leq k \leq m} \Gamma^k) |\Arcset|$.
For mixed-budgeted uncertainty sets, the $\LDij$ values can be easily obtained since $\LDij = \max_{1 \leq k \leq m} L^{\Delta^k}_{ij} $ for every $i,j \in \AllJobs$.
For explicitly-defined uncertainty sets, values $\LDij$ can be computed in polynomial time in the number of scenarii $|S|$, since $\LDij = \max_{s \in S} L_{G(p+\delta^s)}(i,j)$. This relies on the property that $\LDij = L^{\conv(\Delta)}_{ij}$ for any $\Delta$, as proven in \citet{AnchRobPSPHal}.

\bigbreak
In the sequel, the main assumption is that the $\LDij$ values have been precomputed.
Thus they will appear as coefficients of constraints in the proposed mixed-integer programming formulations.

\subsection{Anchored sets}

Let us finally give some preliminaries related to linear formulations for (AnchRob). Recall that a solution of (AnchRob) is a pair $(x,H)$ with a schedule $x$ and a subset $H$ of jobs that is $x$-anchored.
Fix $x$ a schedule of $G(p)$.
It was shown in \citet{AnchRobPSPHal} that a set $H$ is $x$-anchored if and only if $x_j - x_i \geq \LDij$ for every $i \in H \cup \{s\}$, $j \in H$, $i \prec j$.
Let $\GtildeH$ denote the precedence graph obtained from $G(p)$ by introducing additional arcs $(i,j)$, $i \in H \cup \{s\}$, $j \in H$, $i \prec j$ with arc-lengths $\LDij$. Then 

\begin{proposition}[\cite{AnchRobPSPHal}]
\label{prop:caracAnchored}
Let $H \subseteq J$.
Vector $x$ is a schedule of $G(p)$ such that $H$ is $x$-anchored if and only if $x$ is a schedule of $\GtildeH$.
\end{proposition}

Let us now define anchored sets, independently from a given baseline schedule.
A subset $H \subseteq \Jobs$ is \emph{anchored} if there exists a schedule $x$ of $G(p)$ with makespan at most $M$ such that $H$ is $x$-anchored. 
An issue is then to retrieve a baseline schedule $x$ for which $H$ is $x$-anchored. Consider $z$ the earliest schedule of $\GtildeH$: then by Proposition~\ref{prop:caracAnchored}, if $H$ is an anchored set then $H$ is $z$-anchored.
Given $H \subseteq \Jobs$, let $\indic{H}$ denote the incidence vector of $H$.
Let $\Hrond = \{ \indic{H}:\ H \text{ anchored set}\}$, and $\Qrond = \conv(\Hrond)$.
Note that (AnchRob) reduces to finding a max-weight anchored set, i.e., maximizing $\sum_{i \in \Jobs} \aw_i h_i$ for $h \in \Hrond$, or equivalently, for $h \in \Qrond$.

Let us give some definitions related to mixed-integer programming formulations for (AnchRob).
Considered formulations for (AnchRob) involve binary anchoring variables $h \in \BinJobs$ to indicate if jobs are in the anchored set, and continuous variables, say $x \in \Real^q$.
A formulation for (AnchRob) is defined by a polyhedron $\mathcal{P} \subseteq \Real^q \times \CubeJ$ and integrality constraints $h \in \BinJobs$, so that the feasible set of the formulation is $\mathcal{P} \cap (\Real^q\!\times\!\BinJobs)$.
Given $\mathcal{F} \subseteq \Real^q \times \Real^J$, let $\Projh( \mathcal{F} ) = \{ h \in \Real^J:\ \exists (x,h) \in \mathcal{F}\}$ denote its projection on $h$ variables.
A formulation is \emph{valid for (AnchRob)} if $\Projh( \mathcal{P} ) \cap \BinJobs = \Hrond$. Given two polyhedra $\mathcal{P}_1$ and $\mathcal{P}_2$, formulation associated with $\mathcal{P}_1$ is \emph{stronger} than formulation associated with $\mathcal{P}_2$ if $\Projh(\mathcal{P}_1) \subseteq \Projh(\mathcal{P}_2)$.
A formulation \emph{yields a polyhedral characterization for (AnchRob)} if $\Projh(\mathcal{P}) = \Qrond$.
Importantly, if the formulation associated with polyhedron $\mathcal{P}$ yields a polyhedral characterization for a special case of (AnchRob), and $\mathcal{P}$ is described by a polynomial number of inequalities, then the special case of (AnchRob) is polynomial. Indeed (AnchRob) can be solved by the linear program $\max \sum_{i \in J} w_i h_i$ for $(x,h) \in \Prond$.

\section{Linear formulations for (AnchRob)}
\label{sec:formulation}

In this section, we establish linear formulations for (AnchRob) using $\LDij$ values as coefficients. In Section~\ref{sec:form:naive} a naive formulation is given as a benchmark. In Section~\ref{sec:form:domprop} the main dominance property is proven. In Section~\ref{sec:form:dom} a formulation is derived from the dominance property. In Section~\ref{sec:form:tighter} it is compared with other known formulations. 

\subsection{A naive formulation}
\label{sec:form:naive}

Consider schedule continuous variables $x_j$, $j \in \AllJobs$, and anchoring binary variables $h_j \in \{0,1\}$, $j \in \Jobs$. Vector $h$ is the incidence vector of the anchored set $H$. A formulation for (AnchRob) requires constraints to enforce that $H$ is an $x$-anchored set.
The characterization of Proposition~\ref{prop:caracAnchored} suggests a quadratic constraint
\begin{center}
\def\arraystretch{1}
\begin{tabularx}{\textwidth}{r l l X}
\hspace{1.4cm} & $x_j - x_i \geq \LDij h_i h_j$ & \hspace{1.2cm} $\forall i,j \in \AllJobs, i \prec j$ &  \hfill \myeqrefquad\\
\end{tabularx}
\end{center}
to represent precedence constraints of the graph $\GtildeH$.
Note that for dummy jobs $s$ and $t$ there is no $h_j$ decision variable but we set $h_s=1$ and $h_t = 0$ for the ease of notation.
For validity, it is sufficient to check that for each $i \prec j$, if $h_i = 0$ or $h_j = 0$ then inequality \myeqrefquad\ is valid. Indeed it reduces to $x_j - x_i \geq 0$ which holds for every $i \prec j$ for every schedule $x$ of $G(p)$.

\medbreak
Applying a standard linearization technique, a linear formulation (Std) can be obtained by replacing constraint \myeqrefquad\ by the following linear inequality \myeqrefstd
\begin{center}
\def\arraystretch{1}
\begin{tabularx}{\textwidth}{l r l l X}
(Std): &    $\max$ & $\sum_{i \in J} \aw_i h_i$  \\
&     s.t. & $x_j - x_i \geq p_i$ & $\forall (i,j) \in \Arcset$ \\ 
 &         & $x_t \leq M$ \\
 &         & $x_j - x_i \geq \LDij (h_i + h_j - 1)$ & $\forall i,j \in \AllJobs, i \prec j$ & \hfill \myeqrefstd \\
 &         & $x_j \geq 0$ & $\forall j \in \AllJobs$\\
&& $h_j \in \{0,1\}$ & $\forall j \in \Jobs$\\
\end{tabularx}
\end{center}
It is easy to check that when $h_i = 0$ or $h_j=0$, \myeqrefstd\ induces a valid constraint.

\subsection{A dominance property}
\label{sec:form:domprop}

Let $H \subseteq J$ be a subset of jobs. Let us define the set of all baseline schedules $x$ such that $(x,H)$ is feasible for (AnchRob) as $\SetSchedulesH = \{ x \in \SetSchedules:\ H $ $x$-anchored$\}$. 
Note that (AnchRob) problem is to maximize the weight of a set $H$ such that $\SetSchedulesH \not= \varnothing$.
By Proposition~\ref{prop:caracAnchored},
\begin{equation*}
    \SetSchedulesH = \{ x \in \SetSchedules :\hspace{0.3cm} x_j - x_i \geq \LDij\ 
    \hspace{0.3cm} \forall i \in H \cup \{s\}, j \in H, i \prec j\}.
\end{equation*}

Let us now define a set of baseline schedules where the same inequality is imposed, but on pairs $i,j$ with $i \in J \cup \{s\}$:
\begin{equation*}
\SetZSchedulesH = \{ z \in \SetSchedules :\hspace{0.3cm} z_j - z_i \geq \LDij\ \hspace{0.3cm} \forall i \in J \cup \{s\}, j \in H, i \prec j\}.
\end{equation*}

\begin{theorem}
\label{thm:dominance}
Set $\SetZSchedulesH$ is dominant, in the sense that $\SetZSchedulesH \subseteq \SetSchedulesH$, and $\SetSchedulesH\!\not=\!\varnothing$ implies $\SetZSchedulesH\!\not=\!\varnothing$.
\end{theorem}

\begin{proof} 
Since $H \subseteq \Jobs$, the definition of set $\SetZSchedulesH$ contains more constraints than that of set $\SetSchedulesH$. Hence $\SetZSchedulesH \subseteq \SetSchedulesH$.
Let us prove $\SetSchedulesH \not= \varnothing \implies \SetZSchedulesH \not= \varnothing$.
Note that $\SetSchedulesH$ is exactly the set of schedules of $\GtildeH$ with makespan at most $M$.
Let $z$ be the earliest schedule of $\GtildeH$. By assumption, there exists a schedule in $\SetSchedulesH$, thus $z_t \leq M$.
Let $i \in J \cup \{s\}$ and $j \in H$. Let us show that $z_j - z_i \geq \LDij$ holds, even for $i \notin H \cup \{ s\}$. 
Consider a longest \odpath{s}{i}path $P^*_{si}$ in $\GtildeH$, and let $k \in H \cup \{s\}$ be the last vertex of $H \cup \{s\}$ on this path. Then $z_i = L_{\Gtilde}(s,i) = L_{\Gtilde}(s,k) + L_{\Gtilde}(k,i) = z_k + L_{\Gtilde}(k,i)$. The subpath of $P^*_{si}$ from $k$ to $i$ is a longest \odpath{k}{i}path in $\GtildeH$. Since it has no vertex in $H \cup \{s\}$ except $k$, it uses no additional arc and its length is $L_{\Gtilde}(k,i) = \Lz_{ki}$.
It comes $z_k - z_i = - \Lz_{ki}$.
Since $z \in \SetSchedulesH$, $z_j - z_k \geq \LD_{kj}$.
Then $z_j - z_i = (z_j - z_k) + (z_k - z_i) \geq \LD_{kj} - \Lz_{ki}$.
Also, for $k \prec i \prec j$, it holds that $\LD_{kj} \geq \Lz_{ki} + \LD_{ij}$: indeed $\LD_{ij}$ is attained for some $\delta^* \in \Delta \subseteq \RealJp$, hence 
$\LD_{kj} \geq L_{G(p+\delta^*)}(k,i) + L_{G(p+\delta^*)}(i,j) \geq \Lz_{ki} + \LDij$.
Finally $z_j - z_i \geq \LDij$, hence $z \in \SetZSchedulesH$.
\hfill $\square$
\end{proof}

\subsection{Dominance-based linear formulation (Dom)}
\label{sec:form:dom}

 (AnchRob) problem is to maximize the weight of a set $H$ such that $\SetSchedulesH \not= \varnothing$, or equivalently with Theorem~\ref{thm:dominance}, such that $\SetZSchedulesH \not= \varnothing$.
We now introduce a new formulation derived from Theorem~\ref{thm:dominance}, where continuous variables $z$ correspond to a schedule $z \in \SetZSchedulesH$.

\begin{center}
\def\arraystretch{1}
\begin{tabularx}{\textwidth}{l r l l X}
(Dom):&    $\max$ & $\sum_{i \in J} \aw_i h_i$  \\
&     s.t. & $z_j - z_i \geq p_i$ & $\forall (i,j) \in \Arcset$ \\ 
 &         & $z_t \leq M$ \\
 &         & $z_j - z_i \geq \Lz_{ij} + (\LD_{ij} - \Lz_{ij}) h_j$ & $\forall i,j \in \AllJobs,\ i \prec j$ & \hfill \myeqrefdom\\
 &         & $z_j \geq 0$ & $\forall j \in \AllJobs$\\
&& $h_j \in \{0,1\}$ & $\forall j \in \Jobs$\\
\end{tabularx}
\end{center}

\begin{proposition}
\label{prop:MIPzh}
Formulation (Dom) is valid for (AnchRob).
\end{proposition}

\begin{proof}
Let $(z,h)$ be feasible for (Dom), and $H := \{i \in J:\ h_i = 1 \}$. Note first that inequality \myeqrefdom\ implies $z \in \SetSchedulesH$. Hence $z$ is a schedule with makespan at most $M$ and $H$ is $z$-anchored. Hence $h \in \Hrond$.
Conversely, let $h \in \Hrond$ be the incidence vector of an anchored set $H$.
From Prop.~\ref{prop:caracAnchored} and Theorem~\ref{thm:dominance}, there exists $z \in \SetZSchedulesH$. Such a schedule $z$ satisfies $z_j - z_i \geq \LDij$ for every $i \in J \cup \{s \}$ and $j \in H$.
Then $z$ satisfies \myeqrefdom: indeed if $h_j = 1$ inequality \myeqrefdom\ corresponds to inequality $z_j - z_i \geq \LDij$; if $h_j = 0$ then \myeqrefdom\ amounts to $z_j - z_i \geq \Lz_{ij}$, which holds since $z$ is a schedule of $G(p)$.
\hfill $\square$
\end{proof}

Note that precedence constraint $z_j - z_i \geq p_i$ associated with arc $(i,j) \in \Arcset$ is implied by \myeqrefdom\ since $i \prec j$ and $\Lz_{ij} \geq p_i$.

\medbreak
We now introduce a family of valid inequalities. Let $j \in J$. By inequality \myeqrefdom\ with $i=s$, it comes that $z_j \geq \Lz_{sj} + (\LD_{sj} - \Lz_{sj}) h_j$. Since $z \in \SetSchedules$ it satisfies $z_j + \Lz_{jt} \leq M$, thus leading to $M \geq \Lz_{sj} + (\LD_{sj} - \Lz_{sj}) h_j + \Lz_{jt}$. Hence the following valid inequality
\begin{equation}
    \tag{\myeqrefboxsanspar}
    \label{eq:chvatal}
    h_j \leq \left\lfloor \frac{M - (\Lz_{sj} + \Lz_{jt})}{\LD_{sj} - \Lz_{sj}} \right\rfloor.
\end{equation}
In Section~\ref{sec:polyhedralCharac} they will be discussed with respect to the polyhedral characterization under box uncertainty.

\subsection{Comparison with known formulations}
\label{sec:form:tighter}

In this section formulation (Dom) is compared with (Std) and with (Lay), the previously known formulation from \citep{AnchRobPSPHal} dedicated to budgeted uncertainty.

\medbreak
Consider first the continuous relaxations of (Dom), (Std) and the variant of (Std) with quadratic constraint \myeqrefquad. It turns out that they can easily be compared. Indeed for every $h_i, h_j \in [0,1]$, it holds that
$h_j \geq h_i h_j$.
Hence the right-hand side of inequality \myeqrefdom\ is tighter than the right-hand side of inequality \myeqrefquad.
Also for every $h_i, h_j \in [0,1]$, it holds that
$(1 - h_i)(1 - h_j) \geq 0$ or equivalently
$h_i h_j \geq h_i+h_j - 1$.
Hence the right-hand side of inequality \myeqrefquad\ is tighter than the right-hand side of inequality \myeqrefstd.

\medbreak
Let us now investigate the special case of budgeted uncertainty, and compare (Dom) with formulation (Lay) from \citep{AnchRobPSPHal}, that we now recall.
Formulation (Lay) involves anchoring variables $h_j \in \{0,1\}$ for every $j \in J$, and continuous variables $\gam{x}_j$ for every $\gamma \in \{0, \dots, \Gamma\}$, $j \in \AllJobs$.
Formulation (Lay) is based on a so-called \emph{layered graph} $\Glay(h)$ associated with $h \in \CubeJ$ and built as follows. The layered graph $\Glay(h)$ is formed with $\Gamma+1$
copies of the precedence graph called \emph{layers} indexed from $0$ to $\Gamma$.
Let $\gam{i}$ denote the copy of job $i$ in layer $\gamma$. The layered graph features three types of arcs. Horizontal arcs are copies of arcs of $G(p)$, i.e., arcs $(\gam{i}, \gam{j})$ for $(i,j) \in \Arcset$, with length $p_i$. Transversal arcs are $(\gamp{i}, \gam{j})$ for $(i,j) \in \Arcset$, with length $p_i+\dhat_i$. Vertical arcs are $(\gam{i}, \Gam{i})$ for $i \in J$, $\gamma < \Gamma$, with length $-\Mjp(1-h_j)$, where $\Mjp = L_{G(p+\dhat)}(s,j) - L_{G(p)}(s,j)$ for every $j \in J$. It was shown that $h \in \BinJobs$ is the incidence vector of an anchored set if and only if there exists $x$ a schedule of $\Glay(h)$ such that $\Gam{x}_t \leq M$. This leads to the formulation

\begin{center}
\def\arraystretch{1}
\begin{tabularx}{\textwidth}{l r l l l}
(Lay): & $\max$ & $\sum_{i \in \Jobs} \aw_i h_i$ \\
&s.t. & $\gam{x}_j - \gam{x}_i \geq p_i$ & $\forall (i,j) \in \Arcset$, $\forall \gamma \in \Budgets$ & \\
&& $\gam{x}_j - \gamp{x}_i \geq p_i + \dhat_i$ & $\forall (i,j) \in \Arcset$, $\forall \gamma \in \BudgetsLow$ & \\
&& $\Gam{x}_j - \gam{x}_j \geq - \Mjp (1 - h_j)$ & $\forall j \in \Jobs$, $\forall \gamma \in \BudgetsLow$ & \\

& & $\Gam{x}_t \leq M$ & & \\
&& $\gam{x}_j \geq 0$ & $\forall j \in \AllJobs,\ \forall \gamma \in \Budgets$\\
&& $h_j \in \{0,1\}$ & $\forall j \in \Jobs$\\
\end{tabularx}
\end{center}

Variables $x^{\Gamma}_j$, $j \in \AllJobs$ from layer $\Gamma$ can be thought of as a baseline schedule such that $H = \{ i \in J:\ h_i = 1\}$ is $x^{\Gamma}$-anchored. Other continuous variables $x^{\gamma}_j$, $j \in \AllJobs$, $\gamma < \Gamma$ can be regarded as additional variables.

\medbreak
In order to compare linear relaxations of (Lay) and (Dom), we show how to project explicitly those two formulations on the space of $h$ variables. Let $\mathcal{C}$ be the set of \stpath paths in the transitive closure of $G$. Note that there is one-to-one correspondence with chains of the poset $(J, \prec)$.
Let $\PolyLay$ (resp. $\PolyDom$) denote the polytope of solutions $( (\gam{x})_{\gamma \in \{0, \dots, \Gamma\}}, h)$ (resp. $(z,h)$) that are feasible for the continuous relaxation of formulation (Lay) (resp. (Dom)).

\begin{proposition}
\label{prop:projections}
$$\Projh(\PolyDom) = \left\lbrace h \in \CubeJ: \sum_{\substack{(i,j)\in C \\ j \not= t}} \left(\Lz_{ij} + (\LDij - \Lz_{ij}) h_j \right) + \Lz_{j_C t} \leq M\ \forall C \in \mathcal{C} \right\rbrace$$
$$\Projh(\PolyLay) = \left\lbrace h \in \CubeJ: \sum_{\substack{(i,j)\in C \\ j \not= t}} \left(\LDij - \Mjp(1-h_j) \right) + \Lz_{j_C t} \leq M\ \forall C \in \mathcal{C} \right\rbrace$$
where $j_C$ denotes the last vertex of path $C$ before sink $t$.
\end{proposition}

\begin{proof}
Let $h \in \CubeJ$.
Let $\Gbar(h)$ denote the transitive closure of $G$, where arc $(i,j)$ is given the weight $\Lz_{ij} + (\LD_{ij} - \Lz_{ij}) h_j$, with $h_t=0$ for the ease of notation. Then $(z,h) \in \PolyDom$ if and only if $z$ is a schedule of $\Gbar(h)$ with makespan $z_t \leq M$, by definition of formulation (Dom). The existence of such $z$ is equivalent to $L_{\Gbar(h)}(s,t) \leq M$; or equivalently, every \stpath path in $\Gbar(h)$ has length at most $M$. The length of path $C \in \mathcal{C}$ in $\Gbar(h)$ is exactly the left-hand side of the proposed inequality, hence the result.

A similar proof holds for formulation (Lay).
Let $h \in \CubeJ$.
Let $\ell_{\Glay(h)}(R)$ denote the length of a path $R$ in the layered graph $\Glay(h)$.
It holds that $h \in \Projh(\PolyLay)$ if and only if there exists $x$ a schedule of $\Glay(h)$ such that $\Gam{x}_t \leq M$. The existence of $x$ is equivalent to the longest path condition: $L_{\Glay(h)} (\Gam{s}, \Gam{t}) \leq M$; or equivalently, $\ell_{\Glay(h)}(\PathInLay) \leq M$ for every path $\PathInLay$ from $\Gam{s}$ to $\Gam{t}$ in the layered graph $\Glay(h)$.
Let us now show that it is equivalent to $\sum_{{(i,j)\in C,\ j \not= t}} \left(\LDij - \Mjp(1-h_j) \right) + \Lz_{j_C t} \leq M$ for every $C \in \mathcal{C}$.

Assume first $\ell_{\Glay(h)}(\PathInLay) \leq M$ for every path $\PathInLay$ from $\Gam{s}$ to $\Gam{t}$ in $\Glay(h)$. Let $C \in \mathcal{C}$ be an \stpath path in the transitive closure of $G$. Consider an associated \odpath{\Gam{s}}{\Gam{t}} path $\PathInLay^*$ in the layered graph built as follows: for every arc $(i,j) \in C$, $j \not= t$, path $\PathInLay^*$ contains the subpath of length $\LDij$ going from $\Gam{i}$ to a copy $\gam{j}$ of $j$, and the vertical arc $(\gam{j}, \Gam{j})$; path $\PathInLay^*$ also contains the subpath of length $\Lz_{j_C t}$ going from $\Gam{j_C}$ to $\Gam{t}$. Then $\PathInLay^*$ is an \odpath{\Gam{s}}{\Gam{t}}path in $\Glay(h)$, by assumption it has length at most $M$, hence $\sum_{{(i,j)\in C,\ j \not= t}} \left(\LDij - \Mjp(1-h_j) \right) + \Lz_{j_C t} \leq M$.

Conversely assume $\sum_{{(i,j)\in C,\ j \not= t}} \left(\LDij - \Mjp(1-h_j) \right) + \Lz_{j_C t} \leq M$ for every $C \in \mathcal{C}$. Let $\PathInLay$ be an \odpath{\Gam{s}}{\Gam{t}}path in $\Glay(h)$. Let $C^*$ be the path defined the successive jobs $j$ such that $\PathInLay$ contains a vertical arc $(\gam{j}, \Gam{j})$; then $C^* \in \mathcal{C}$. For every $(i,j) \in C^*$, let $R^{(i,j)}$ denote the subpath of $R$ from $\Gam{i}$ to a copy $\gam{j}$ of job $j$. Since $R^{(i,j)}$ uses at most $\Gamma$ transversal arcs, it comes $\ell_{\Glay(h)}(R^{(i,j)}) \leq \LDij$. Similarly, let $R^{(j_{C^*},t)}$ denote the subpath of $R$ between $\Gam{j_{C^*}}$ and $\Gam{t}$, then $\ell_{\Glay(h)}(R^{(j_{C^*},t)}) \leq \Lz_{j_{C^*} t}$.
The total length of $R$ is $\ell_{\Glay(h)}(R) = \sum_{(i,j) \in C^*, j \not= t} (\ell_{\Glay(h)}(R^{(i,j)}) - \Mjp(1-h_j)) + \ell_{\Glay(h)}(R^{(j_{C^*},t)})$, thus upper-bounded by $\sum_{(i,j) \in C^*, j \not= t} (\LDij - \Mjp(1-h_j)) + \Lz_{j_{C^*} t}$. By assumption this at most $M$ since $C^* \in \mathcal{C}$, hence $\ell_{\Glay(h)}(\PathInLay) \leq M$. This proves the claimed result.
\hfill $\square$
\end{proof}

\begin{proposition}
\label{prop:DomDomineFormX}
If the instance satisfies
    $L_{G(p+\dhat)}(s,j) - \Lzsj \geq \LDij - \Lzij$ for every $i \prec j$,
then
formulation (Dom) is stronger than formulation (Lay), in the sense that $\Projh(\PolyDom) \subseteq \Projh(\PolyLay)$.
\end{proposition}

\begin{proof}
For every $i \prec j$, it holds that
$\LDij - D_j (1-h_j) \leq \Lzij + (\LDij - \Lzij) h_j$. Indeed it is equivalent to $(D_j - \LDij + \Lzij ) h_j \leq D_j - \LDij + \Lzij$. This latter inequality is satisfied for every $h_j \in [0,1]$, since $D_j - \LDij + \Lzij \geq 0$ by the assumption.
Using the explicit definition of projections in Prop.~\ref{prop:projections}, it comes that the inequalities defining the projection of (Dom) are tighter than the inequalities defining the projection of (Lay). \hfill $\square$
\end{proof}

Importantly, the assumption of Proposition~\ref{prop:DomDomineFormX} is satisfied if the precedence graph is critical. Indeed, assume $G(p)$ is critical and let $i \prec j$. The $D_j$ value satisfies $D_j \geq \LD_{sj} - \Lzsj \geq \Lz_{si} +\LD_{ij} - \Lzsj$. Since $G(p)$ is critical, it holds that $\Lz_{si} + \Lz_{ij} = \Lz_{sj}$, hence $D_j \geq \LDij - \Lzij$.

We now give an example satisfying the assumption of Proposition~\ref{prop:DomDomineFormX}, and where $\Projh(\PolyDom)$ is strictly included in $\Projh(\PolyLay)$. Let $J = \{1, 2, 3\}$ be a set of three jobs, let $G$ be the path $(s, 1, 2, 3, t)$, and let $p_i = 1$ and $\dhat_i = 1$ for every $i \in J$. Let also $\Gamma=1$, and $M = 3$. Consider $h^* = (1, 0, \frac{1}{2})$. To see that $h \in \Projh(\PolyLay)$, consider the layered graph $\Glay(h^*)$, represented in Figure~\ref{fig:exlay}.
The vector $x$ defined by $x^1 = (0, 0, 1, 2, 3)$ in layer $1$, $x^0 = (0, 0, 2, 3, 4)$ in layer $0$ is also represented into brackets on the vertices in Figure~\ref{fig:exlay}. 
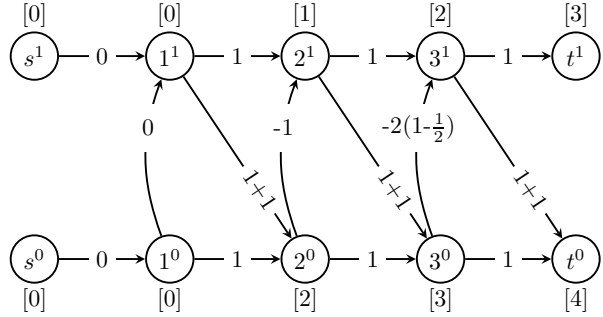
\begin{figure}[H]
    \centering
    \scalebox{0.9}{
        \input{exlayered.tex}
    }
    \caption{Layered graph $\Glay(h^*)$ for $h^*=(1,0,\frac{1}{2})$.}
    \label{fig:exlay}
\end{figure}

It can be checked that vector $x$ is a schedule of $\Glay(h^*)$, hence $(x,h^*) \in \PolyLay$. By contrast, vector $h^*$ violates inequality defining $\Projh(\PolyDom)$ associated with path $P = (s,3,t)$. Indeed $\Lz_{s3} = 2$, $\LD_{s3}=3$, $\Lz_{3t}=1$ and the inequality is $2 + h_3 + 1 \leq 3$. Thus $h^* \notin \Projh(\PolyDom)$, and $\Projh(\PolyDom) \subsetneq \Projh(\PolyLay)$. More precisely, it is a case where (Dom) yields a polyhedral characterization of (AnchRob), as shown next in Section~\ref{sec:UAnchRob}.

\medbreak
The question whether (Dom) is stronger than (Lay) for any instance, is left open. However in numerical experiments presented in Section~\ref{sec:num}, the linear bound of (Dom) was always better than the linear bound of (Lay).

\section{Polyhedral characterization for special cases}
\label{sec:polyhedralCharac}

In this section, we provide polyhedral characterizations of (AnchRob) for two special cases: box uncertainty and 1-disruption uncertainty for critical precedence graph.

\subsection{Box uncertainty}

This section is devoted to box uncertainty, which is the special case of budgeted uncertainty where $\Gamma = |\Jobs|$, and thus $\dhat \in \Delta$.
Note first that if $\xbas$ the earliest schedule of $G(p+\dhat)$ satisfies the deadline constraint $\xbas_t \leq M$, then all jobs can be anchored and $(\xbas, J)$ is an optimal solution of (AnchRob). However in general the schedule $\xbas$ may have a makespan larger than $M$, hence not all jobs can be anchored (see, e.g., the example in Figure~\ref{fig:scheduleBox} in Section~\ref{sec:prelim}).

In \citet{AnchRobPSPHal} a polynomial algorithm for (AnchRob) under box uncertainty was provided.
This algorithm is as follows:
compute $\xbas$ the earliest schedule of $G(p+\dhat)$ and $\xhaut$ the latest schedule of $G(p)$ such that $\xhaut_t=M$; let $H^* = \{i \in \Jobs\ |\ \xbas_i \leq \xhaut_i\}$ and $x_i = \min \{\xbas_i, \xhaut_i \}$ for every $i \in \AllJobs$; return $(x, H^*)$.
\medbreak
The main result is that in this case the polytope $\Qrond$ of anchored sets is characterized by inequalities \eqref{eq:chvatal}. 

\begin{theorem}
\label{thm:box}
For box uncertainty,
$$\Qrond = \{ h \in \CubeJ:\ h \text{ satisfies \eqref{eq:chvatal}}\}.$$
\end{theorem}

\begin{proof}
Consider schedules $\xbas$ and $\xhaut$, defined by $\xbas_j = L_{G(p+\dhat)}(s,j)$ and $\xhaut_j = M - \Lz_{jt}$ for every $j \in J$.
Note first that for box uncertainty, $\LD_{sj} = L_{G(p+\dhat)}(s,j) =\xbas_j$. Inequality \eqref{eq:chvatal} writes in this case: $h_j \leq \left\lfloor \frac{M - (\Lz_{sj} + \Lz_{jt})}{\LD_{sj} - \Lz_{sj}} \right\rfloor = \left\lfloor \frac{\xhaut_j - \Lz_{sj}}{ \xbas_j - \Lz_{sj}} \right\rfloor$. Equivalently, it implies $h_j \leq 1$ if $j \in H^*$, and $h_j \leq 0$ if $j \notin H^*$.
Hence $\{ h \in \CubeJ:\ h \text{ satisfies \eqref{eq:chvatal}}\} = \BoxHStar$.

Every extreme point of $\BoxHStar$ is the incidence vector of a set $H \subseteq H^*$. Since any subset of an anchored set is anchored, and $H^*$ is anchored, then $\BoxHStar \subseteq \Qrond$.
Conversely, if $H$ is an anchored set, its incidence vector $\indic{H}$ satisfies valid inequalities \eqref{eq:chvatal}: thus $H \subseteq H^*$. Hence $\Qrond \subseteq [0,1]^{H^*}\times \{0\}^{J \setminus H^*}$.
\hfill $\square$
\end{proof}

Note that Theorem~\ref{thm:box} also holds for any $\Delta$ that has a greatest element $\dhat$. This case is more general than $\Delta$ being a box, as mentioned in Section~\ref{sec:uncSets}.

\subsection{1-disruption uncertainty}

In this section, set $\Delta$ is a 1-disruption uncertainty set, i.e., a budgeted uncertainty set with $\Gamma=1$ and $\dhat_i = \dhat_0$ for every $i \in J$. In Section~\ref{sec:UAnchRob} the polyhedral characterization is shown for the special case with zero processing times. In Section~\ref{sec:critical} it is extended to any critical precedence graph with non-zero processing times.

\subsubsection{The Unitary (AnchRob)}
\label{sec:UAnchRob}
Assume first $p_i = 0$ for every $i \in J$.
W.l.o.g. $\dhat_0 = 1$. Finally assume deadline $M$ integer.
If $M$ is not integer, it can be replaced with $\lfloor M \rfloor$ w.l.o.g.: indeed, since $p$ and $\dhat$ are integer, for any $H$ the earliest schedule of $\Gtilde_H$ is integer-valued with integer makespan.

In this special case, an instance is the precedence graph $G=(\AllJobs, \Arcset)$, $p=0$, $\dhat_0=1$, integer deadline $M$, and anchoring weights. This special case is referred to as the Unitary Anchor-Robust (U-AnchRob) problem.
For unit anchoring weights, it has been identified as a polynomial case in \citet{AnchRobPSPHal} by an equivalence with a problem on posets for which a min/max theorem is known \citep{schrijver}.

\medbreak
For (U-AnchRob) our main result is a characterization of the polytope through formulation (Dom). It relies on the polyhedral result of Theorem~\ref{thm:integralityUAnchRob}. Let $(\AllJobs, \prec)$ be a poset with $s$ (resp. $t$) a least (resp. greatest) element. Consider the inequalities
\begin{center}
\def\arraystretch{1}
\begin{tabular}{ l@{\hspace{0.5cm}}l @{\hspace{0.5cm}} c}
     $z_i - z_s \geq 0$ & $\forall i \in \Jobs$ & \ineqsi \\
    $z_t - z_i \geq 0$ & $\forall i \in \Jobs$ & \ineqit \\
    $z_t \leq M$ && \ineqleqM \\
    $z_j - z_i \geq h_j$ & $\forall i,j \in \Jobs$, $i \prec j$ & \ineqij  \\
     $z_j \geq 0$ & $\forall j \in \AllJobs$ & \ineqznonneg\\
     $h_j \leq 1$ & $\forall j \in \Jobs$ & \ineqhun \\
     $h_j \geq 0$ & $\forall j \in \Jobs$ & \ineqhnonneg \\
\end{tabular}
\end{center}
and let $\Prond = \{ (z,h) \in \Real^{\AllJobs}\times \Real^J:\ \text{(a)} - \text{(g)}\}$.

\begin{theorem}
\label{thm:integralityUAnchRob}
The polytope $\Prond$ is integer.
\end{theorem}

\begin{proof}
Consider the polytope $\Prond$ formed with all pairs $(z,h) \in \RealAllJp \times \RealJp$ satisfying the constraints \ineqsi--\ineqhnonneg.
To prove integrality of $\Prond$, the main idea is to define an auxiliary extended polyhedron where $h$ variables can be expressed linearly from $z$ variables and additional $z'$ variables.
Let us define the auxiliary polyhedron $\Prond'$ formed with all triplets $(z,z',h) \in \RealAllJp \times \RealAllJp \times \Real^{\Jobs}$ satisfying the constraints
\begin{center}
\def\arraystretch{1}
\begin{tabular}{ l@{\hspace{0.5cm}} l @{\hspace{0.5cm}}c}
     $z_j - z_i \geq 0$ & $\forallarcij$ & \ineqAuxHoriz  \\
     $z'_j - z'_i \geq 0$ & $\forallarcij$ & \ineqAuxHorizBas \\
     $z'_j - z_i \geq 1$ & $\forallarcij$ & \ineqAuxTrans \\
     $z_j - z'_j \geq -1$ & $\forall j \in \Jobs$ & \ineqAuxVertMinus \\
     $z_j - z'_j \leq 0$ & $\forall j \in \Jobs$ & \ineqAuxVert \\
     $z_t \leq M$ && \ineqAuxleqM \\     
     $h_j = 1 + z_j - z'_j$ & $\forall j \in \Jobs$ & \eqAuxH \\
\end{tabular}
\end{center}
Let us prove the following claim. \emph{Claim 1. $\Prond = \Projzh(\Prond')$.}\\ First, given $(z,h) \in \Prond$, let us prove the existence of $z'$ such that $(z,z',h) \in \Prond'$. Define $z'_j = 1 + z_j - h_j$ for every $j\in \Jobs$, and $z'_s = z_s$, $z'_t = 1 + z_t$. Then the (in)equalities \ineqAuxHoriz\ to \eqAuxH\ can be checked for the triplet $(z,z',h)$ as follows.\\
-- \ineqAuxHoriz\ and \ineqAuxleqM\ hold by assumption on $z$,
and \eqAuxH\ by definition of $z'$;\\
-- \ineqAuxHorizBas: if $j \not= t$ then $z'_j - z'_i = z_j - z_i - h_j + h_i \geq h_i \geq 0$ by \ineqij; if $j=t$ then $z'_j - z'_i = z_t - z_i + h_i \geq h_i \geq 0$ by \ineqit;\\
-- \ineqAuxTrans: if $j \not= t$ then $z'_j - z_i = 1 + z_j - h_j - z_i \geq 1$ by \ineqij; if $j=t$, follows from \ineqit;\\
-- \ineqAuxVertMinus\ and \ineqAuxVert: $z_j - z'_j = -1 + h_j \in [-1, 0]$ since $h_j \in [0,1]$. \\
Conversely, let $(z,z',h) \in \Prond'$. Let us check inequalities \ineqsi--\ineqhnonneg\ for $(z,h)$.\\
-- \ineqsi\ holds by sum of \ineqAuxHoriz\ along an \odpath{s}{i}path;\\
-- \ineqit\ holds by sum of \ineqAuxHoriz\ along an \odpath{i}{t}path; \\
-- \ineqleqM\ is clear;\\
-- \ineqij: for $i \prec j$ we have $z_j - z_i = z_j - z'_j + z'_j - z_k + z_k - z_i$, where $k$ is the last vertex distinct from $j$ on a path from $i$ to $j$ in the precedence graph (possibly $k=i$). Then $(k,j) \in \Arcset$ so $z'_j - z_k \geq 1$ by \ineqAuxTrans. Also $z_k - z_i \geq 0$ by summing \ineqAuxHoriz\ along the \odpath{i}{k} path. Hence $z_j - z_i \geq z_j - z'_j + 1 + 0 = h_j$ by \eqAuxH. Hence \ineqij\ is satisfied; \\
-- \ineqhun\ (resp. \ineqhnonneg) comes from \eqAuxH\ and \ineqAuxVertMinus\ (resp. \eqAuxH\ and \ineqAuxVert).\\
This completes the proof of Claim 1. $\diamond$

Now we prove: \emph{Claim 2. $\Prond'$ is integer.}\\
 Let $(z,z',h)$ be an extreme point of $\Prond'$. It satisfies the $n$ equalities \eqAuxH\, and it saturates $2(n+2)$ linearly independent inequalities among \ineqAuxHoriz--\ineqAuxleqM. Thus $(z,z')$ is an extreme point of $ \{ (z,z') \in \RealAllJp \times \RealAllJp:\ \text{\ineqAuxHoriz-\ineqAuxleqM} \}$. The constraint matrix of \ineqAuxHoriz--\ineqAuxleqM\ is totally unimodular, and the right-hand side is integer since $M$ is integer. Hence $(z,z')$ is integer and so is $(z, z',h)$. This ends the proof of Claim 2. $\diamond$

By Claim 1, $\Prond = \Projzh(\Prond')$. It holds that
for any extreme point $(z,h)$ of $\Prond$ there exists $z'$ such that $(z,z',h)$ is an extreme point of $\Prond'$. By Claim 2, such $(z,z',h)$ is integer, hence $(z,h)$ is integer. This proves the integrality of $\Prond$. \hfill $\square$
\end{proof}

Theorem~\ref{thm:integralityUAnchRob} implies the following polyhedral characterization

\begin{proposition}
\label{prop:caracUA}
Formulation (Dom) yields a polyhedral characterization for (U-AnchRob).
\end{proposition}

\begin{proof}
Let us prove that the linear relaxation of (Dom) correponds to the polytope $\Prond$ from Theorem~\ref{thm:integralityUAnchRob}.
For (U-AnchRob) for every $i,j \in \AllJobs$ such that $i \prec j$ we have $\Lz_{ij} = 0$ since $p=0$. The worst-case longest paths values are as follows. For every $j\in \Jobs$, $\LD_{sj} = 0$ if $j$ has no predecessor except $s$, and $\LD_{sj} = 1$ otherwise. For every $i \prec j$ with $i \not= s$, $\LDij = 1$.
For a pair $i \prec j$ with $i \not= s$, inequality \myeqrefdom\ thus writes $z_j - z_i \geq h_j$, which is inequality \ineqij\ or \ineqit\ from the definition of $\Prond$.
For a pair $s \prec j$, inequality \myeqrefdom\ writes $z_j - z_s \geq 0$ if $j$ has no predecessor in $J$, and $z_j - z_s \geq h_j$ otherwise. If $j$ has no predecessor in $J$, it is inequality \ineqsi\ from the definition of $\Prond$. Otherwise $j$ has a predecessor $k \in J$ and inequality $z_j - z_s \geq h_j$ is dominated by inequalities $z_k - z_s \geq 0$ and $z_j - z_k \geq h_k$. Hence it is satisfied by any element of $\Prond$. Thus $\PolyDom = \Prond$.

By Theorem~\ref{thm:integralityUAnchRob}, $\Prond$ is integer, hence $\Projh(\Prond)$ is integer. Namely, each extreme point of $\Projh(\Prond)$ is the incidence vector of an anchored set. Thus $\Projh(\PolyDom) = \Projh(\Prond) = \Qrond$, and (Dom) yields a polyhedral characterization for the problem. \hfill $\square$
\end{proof}

Note that Proposition~\ref{prop:caracUA} implies also the integrality of schedule variables $z$, which are not required to be integer in general.

\medbreak
As a corollary, a complexity result is that Theorem~\ref{thm:integralityUAnchRob} generalizes the polynomial case of (U-AnchRob) proven in \citet{AnchRobPSPHal} for unit anchoring weights, to any non-negative anchoring weights.

\medbreak
The polyhedral characterization from (Dom) can be projected out to obtain a complete description of $\Qrond$ using $h$ variables only. Namely, the inequalities from the projection stated in Proposition~\ref{prop:projections} are $\sum_{i \in C} h_i \leq M$ for every chain $C$ of the subposet $(J^*, \prec)$, where $J^*$ is the set of jobs with at least a predecessor different from $s$. It is a family of inequalities with exponential size.
In that regard, (Dom) is a compact extended formulation with additional variables $z$: with $O(n)$ continuous additional variables, it is possible to describe the polytope $\Qrond$ with a polynomial $O(n^2)$ number of inequalities.

\subsubsection{Critical precedence graphs}
\label{sec:critical}

Let us now consider processing times $p \in \RealJp$. Recall that the precedence graph $G(p)$ is \emph{critical} if the length of all \stpath paths is the same.
A case where $G(p)$ is critical is when processing times are equal to zero, since the length of all \stpath paths is zero. For non-zero processing times, it can also be proven that for some precedence graphs, if $G(p)$ is quasi-critical then it is critical. This holds for series-parallel precedence graphs; the result is proven in \ref{app:sp} and used for numerical results in Section~\ref{sec:num} to generate critical precedence graphs.
\medbreak
Let us now consider a critical precedence graph $G(p)$, and extend the polyhedral characterization result obtained in Section~\ref{sec:UAnchRob}.
It is assumed that $M = \Lz_{st} + \dhat_0 M'$ with $M'$ integer. Indeed $M$ can be tightened to $\Lz_{st} + \dhat_0 \left\lfloor {(M - \Lz_{st})}/{\dhat_0} \right\rfloor$ w.l.o.g. Under this assumption we prove that

\begin{theorem}
\label{thm:characCritical}
For 1-disruption uncertainty and critical precedence graph, formulation (Dom) yields a polyhedral characterization of (AnchRob).
\end{theorem}

\begin{proof}
Let $\mathcal{I}$ denote an instance of (AnchRob) for 1-disruption uncertainty with deviation $\dhat_0$, critical precedence graph $G(p)$, and deadline $M$.
Let $\zref$ be the earliest schedule of $G(p)$, namely $\zref_i = \Lz_{si}$ for every $i \in \AllJobs$.
Since $G(p)$ is critical, for every $i \prec j$ it holds that  
$\Lz_{si} + \Lz_{ij} + \Lz_{jt} = \Lz_{sj} + \Lz_{jt}$, hence $\Lzij = \zref_j - \zref_i$. Also $\LDij = \zref_j - \zref_i + \dhat_0$.
Inequalities \myeqrefdom\ then write
$z_j - z_i \geq (\zref_j - \zref_i) + \dhat_0 h_j$, or equivalently
$\frac{z_j - \zref_j}{\dhat_0} - \frac{z_i - \zref_i}{\dhat_0} \geq h_j.$
Similarly the deadline constraint is equivalent to
$\frac{z_t - \zref_t}{\dhat_0} \leq \frac{M - \zref_t}{\dhat_0}.$
Consider a new instance $\mathcal{I}'$ of (U-AnchRob) defined by: the precedence graph $G$, zero processing times, and $M' = \frac{M - \Lz_{st}}{\dhat_0}$. Then $M'$ is integer by assumption on $M$.
Solution $(z,h)$ is feasible for (Dom) in instance $\mathcal{I}$ if and only if solution $(\frac{z - \zref}{\dhat_0}, h)$ is feasible for (Dom) in instance $\mathcal{I}'$.
Hence if $(z,h)$ is an extreme point of $\PolyDom$ then $(\frac{z - \zref}{\dhat_0}, h)$ is extreme for the (U-AnchRob) instance. By Proposition~\ref{prop:caracUA}, vector $h$ is integer. Hence the claimed result.  \hfill $\square$
\end{proof}
Theorem~\ref{thm:characCritical} thus yields a polyhedral characterization of polynomial size for this special case. We mention that this is a polynomial case of (AnchRob) that was not identified in \citet{AnchRobPSPHal}.

\section{Numerical results}
\label{sec:num}

We investigate the impact of theoretical polyhedral results from Section~\ref{sec:polyhedralCharac} on the performance of formulations (Lay), (Std), (Dom) for various instance classes.
In Section~\ref{sec:num:instances}, instances and settings are presented.
Section~\ref{sec:num:budget} is dedicated to budgeted uncertainty, and Section~\ref{sec:num:others} to partition-budgeted and mixed-budgeted uncertainty sets.

\subsection{Instances and settings}
\label{sec:num:instances}

Instances are randomly generated as follows.
We consider instance classes with a four-field label \inst{F1\_F2\_F3\_F4}:

\begin{itemize}
\setlength\itemsep{0em}
    \item Field \inst{F1} concerns precedence graph $G$
    \begin{itemize}
        \item \inst{ER:} precedence graphs randomly generated according to Erdos-Renyi model, i.e., arc $(i,j)$ is in $G$ with probability $pr = 10/n$.
        \item \inst{SP:} Series-Parallel precedence graphs, inductively generated by drawing randomly series or parallel compositions.
    \end{itemize}
    
    \item Field \inst{F2} denotes processing times $p$
    \begin{itemize}
        \item \inst{pZero:} $p_i = 0$ for every $i \in J$;
        \item \inst{pRand:} $p_i$ is randomly generated in range $[5, 20]$;
        \item \inst{pQCri:} $p$ is obtained by applying the following procedure: start from the values generated for class \inst{pRand}; increase the processing time of a randomly selected job until every job is on a critical path. Hence $G(p)$ is quasi-critical.
    \end{itemize}
    
    \item Field \inst{F3} denotes deviation $\dhat$
    \begin{itemize}
        \item \inst{dRand:} for instances \inst{pRand} and \inst{pQCri}, $\dhat_i$ is randomly generated in $[1, \frac{1}{2} p_i]$ for every $i \in J$; for instances \inst{pZero}, $\dhat$ is equal to the values generated for instances \inst{pQCri};
        
        \item \inst{dUnif:} $\dhat_i = \dhat_0$ for every $i \in J$. Value $\dhat_0$ is randomly selected in the deviation values of instances \inst{dRand}.
    \end{itemize}
    
    \item Field \inst{F4} denotes the uncertainty set
    \begin{itemize}
        \item \inst{$\Gamma$1}, \inst{$\Gamma$2}, \inst{$\Gamma$3} correspond to budgeted uncertainty with deviation defined by \inst{F3} and $\Gamma=1,2,3$ respectively.
       
        \item \inst{Partition}: jobs are partitioned into two subsets $J^1$ and $J^2$, every job being in $J^1$ with probability 0.75. Given the deviation $\dhat$ defined by \inst{F3}, 
        the deviation vector of the \inst{Partition} instance is $\lfloor 0.1 \dhat_i \rfloor$ for every $i \in J^1$, and $\dhat_i$ for every $i \in J^2$. Budgets are $\Gamma^1 = 10$ and $\Gamma^2=1$.
        
        \item \inst{Mixed}: $\Delta = \Delta^1 \cup \Delta^2$ where $\Delta^1$ has deviation $\dhat$ defined by \inst{F3} and $\Gamma^1=1$, and $\Delta^2$ has deviation $\tau \dhat$, with $\tau=0.2$ and $\dhat$ defined by \inst{F3} and $\Gamma^2=10$.
        
    \end{itemize}
    
\end{itemize}

Each label corresponds to a class of 10 instances.
For completeness, a formal definition of series-parallel precedence graphs is given in \ref{app:sp}. It is also proven in the Appendix that instance classes \inst{SP\_pQCri} yield critical precedence graphs.

\medbreak
The number of jobs is set to $n=300$. Anchoring weights are unitary. 
Deadline $M$ is set to $M^{\frac{1}{2}} = \frac{1}{2}(L_{G(p)}(s,t) + L_{G(p+\dhat)}(s,t))$, that is, it is halfway between the min makespan of any schedule of $G(p)$ and the min makespan of a static-robust schedule.
Unreported results showed that choosing a deadline other than $M^{\frac{1}{2}}$ leads to similar results in terms of the comparison of formulations.
For budgeted uncertainty, the budget is $\Gamma \in \{1,2,3\}$.
The choice of a small uncertainty budget was previously motivated in the literature, see, e.g., \citep{AnchRobPSPHal,HerroelenStablePreschedule}.

\medbreak
 For each instance, formulations (Dom), (Std), and (Lay) have been implemented using Julia 0.6.2, JuMP 0.18.5. 
Mixed-integer programs are solved with CPLEX 12.8 on a PC under Windows 10 with Intel Core i7-7500U CPU 2.90GHz and 8 Go RAM.
The time limit is 300 seconds.
\medbreak
The valid inequalities \eqref{eq:chvatal} appear to be added by CPLEX on the fly.
Thus they are not hardcoded in formulations.

\subsection{Impact of instance parameters for budgeted uncertainty}
\label{sec:num:budget}

Let us first investigate the case of budgeted uncertainty.
Table~\ref{tab:ER} and Table~\ref{tab:SP} present the results for \inst{ER} and \inst{SP} instances respectively.
Each table presents results relative to 8 instance classes: the first 6 instance classes are with $\Gamma = 1$ and with all combinations of processing times and deviations, and the last 2 instance classes are with $\Gamma=2$ and $\Gamma=3$.
For each instance class, checkmarks in the first three columns indicate if the assumptions of Theorem~\ref{thm:characCritical} are matched:
\smallbreak
\def\arraystretch{1}
\begin{tabular}{l l}
-- crit.: & the precedence graph $G(p)$ is critical;\\
-- unif.: & deviation $\dhat$ is uniform;\\
-- \inst{$\Gamma$1}: & $\Gamma=1$.\\
\end{tabular}
\smallbreak
The tables feature:
\smallbreak
\begin{tabular}{l l}
-- opt: & average optimal value for instances solved optimally;\\
\multicolumn{2}{l}{-- for each formulation (Lay), (Std) and (Dom):}\\
\indent -- \#solved: & number of instances, out of 10, solved optimally within the time limit;\\
\indent -- gap: & average final gap of unsolved instances; \\
\indent -- time: & average computation time for solved instances in seconds;\\
\indent -- LPGap: & average gap $\frac{b - opt}{opt}$ between integer optimum $opt$ and linear bound $b$;\\
\indent -- CPXGap: & average gap obtained by CPLEX at root node.\\
\end{tabular}
\medbreak
The computation times do not include the preprocessing time for computing $\LDij$ values. The computation is done by a dynamic programming algorithm linear in $\Gamma |\Arcset|$. Its running time is negligible with respect to MIP computation time: on average 0.153 seconds for \inst{ER} instances and 0.161 seconds for \inst{SP} instances.

\bigbreak

Let us now comment on the impact of the instance parameters.

\medbreak
\emph{Polyhedral characterization cases.}
Instance classes
\inst{ER\_pZero\_dUnif\_$\Gamma$1}, \inst{SP} \inst{\_pZero\_dUnif\_$\Gamma$1}, and \inst{SP\_pQCri\_} \inst{dUnif\_$\Gamma$1}, correspond to polyhedral characterization cases. As expected, formulation (Dom) solves the problem in less than one second and LPGap = 0\%. Formulation (Lay) has non-zero LPGap, but CPLEX adds suitable cuts to close the gap at root node.
\medbreak
\emph{Impact of uniform deviation.}
Consider now the 6 instance classes \inst{F1\_F2\_dUnif}\inst{\_$\Gamma$1}. These are the first three row entries of Table~\ref{tab:ER} and the first three row entries of Table~\ref{tab:SP}.
On these 6 instance classes, formulations (Dom) and (Lay) still behave well. (Dom) solves all instances in less than one second. In particular, its LPGap is still very small: at most 0.42\%. (Lay) also performs well, the LPGap of (Lay) is larger, but CPXGap is comparable for (Dom) and (Lay).
We note that uniform deviation has an important impact on the performance of formulations. Consider, e.g., instance classes \inst{SP\_pZero\_dUnif\_$\Gamma$1} and \inst{SP\_pZero\_dRand\_$\Gamma$1}. For uniform deviation (Dom) is integer ; for non-uniform deviation it has $9.70\%$ CPXGap and solves only 4 instances out of 10 within the time limit.
By contrast with (Dom) and (Lay) on these instances, formulation (Std) performs very poorly and solves only 41 instances out of 120 (vs. 107 out of 120 for (Dom)).

\medbreak
\emph{Impact of the precedence graph.} The impact of the precedence graph being critical is limited, as shown for example by the comparison of instances \inst{ER\_pZero} and \inst{ER\_pQCri}. Both are efficiently solved, while the precedence graph is critical for the former, and not critical for the latter. Even more, instances with \inst{pRand} appear to be easy instances, while they do not have critical precedence graphs. An interpretation is that for such instances, a large number of jobs are not on critical paths and thus they can be anchored; note, e.g., that the optimal value is greater for \inst{pRand} instances than for others.
\medbreak
\emph{Impact of uncertainty budget.}
When $\Gamma$ is increased, the performance of (Lay) deteriorates. It gets even worse than (Std) for \inst{ER} instances, see \inst{ER\_pZero\_dUnif\_$\Gamma$3} where (Std) solves 3 instances and (Lay) solves 1 instance. Importantly, the size of formulation (Lay) increases with $\Gamma$. For (Dom) and (Std) only the values of the coefficients $\LDij$ depend on the budget, and not the size of the formulation.

\begin{table}[H]
\footnotesize
\input{tab_classical_ER}
\caption{\inst{ER} instances, budgeted uncertainty}
\label{tab:ER}
\end{table}

\begin{table}[H]
\footnotesize
\input{tab_classical_SP}
\caption{\inst{SP} instances, budgeted uncertainty}
\label{tab:SP}
\end{table}

\subsection{Beyond budgeted uncertainty}
 \label{sec:num:others}
 
Let us now present results when $\Delta$ is an uncertainty set with several budgets.

\subsubsection{Partition-budgeted uncertainty set}

Table~\ref{tab:groupsER} and Table~\ref{tab:groupsSP} give computational results for \inst{ER} and \inst{SP} instances under \inst{Partition} uncertainty.
In this case, the $\LDij$ values were computed by a dynamic programming algorithm with complexity linear in $\Gamma^1 \Gamma^2 |\Arcset|$.
The computation was done in 4.364 seconds on average for \inst{ER} instances, and 4.916 seconds on average for \inst{SP} instances.

It comes that (Dom) solves all 60 \inst{ER} instances, and 54 \inst{SP} instances, which is better than for budgeted uncertainty. For \inst{SP} instances, the results are comparable to those of budgeted uncertainty. By contrast, \inst{ER} instances appear to be very easy for this uncertainty set. 
Hence solving the MIP formulation (Dom) for \inst{Partition} uncertainty does not seem harder than for budget \inst{$\Gamma$1}. This highlight that (Dom) can be readily used to handle several budget constraints.

\begin{table}[H]
\footnotesize
\input{tab_groups_ER}
\caption{\inst{ER} instances, \inst{Partition} uncertainty}
\label{tab:groupsER}
\end{table}

\begin{table}[H]
\footnotesize
\input{tab_groups_SP}
\caption{\inst{SP} instances, \inst{Partition} uncertainty}
\label{tab:groupsSP}
\end{table}

\subsubsection{Mixed-budgeted uncertainty sets}

Table~\ref{tab:unionER} and Table~\ref{tab:unionSP} give computational results for \inst{ER} and \inst{SP} instances under \inst{Mixed} uncertainty.
In this case $\Delta = \Delta^1 \cup \Delta^2$ where $\Delta^1$ is the uncertainty set corresponding to instance classes with fourth field \inst{$\Gamma$1}. The $\LDij$ values were precomputed by the same dynamic programming algorithm as for budgeted uncertainty. This computation was done in 0.152 seconds on average for \inst{ER} instances, and 0.154 seconds on average for \inst{SP} instances.

First, (Dom) solves optimally 59 \inst{ER} instances and 42 \inst{SP} instances for \inst{Mixed}, in comparison with 59 \inst{ER} instances and 48 \inst{SP} instances for budgeted uncertainty ($\Gamma=1$). That is, the performance of the formulation is not very sensitive to the change of uncertainty set.

Some conclusions given in Section~\ref{sec:num:budget} also hold for \inst{Mixed} uncertainty. Namely, instances with \inst{dUnif} are easier than instances with \inst{dRand}. It can be related to the influence of uniform deviation on the performance of (Dom), in connection with the polyhedral characterization result.

The optimal number of anchored jobs is often the same for $\Delta$ or $\Delta^1$. Namely, the value of opt can be compared between Table~\ref{tab:ER} and Table~\ref{tab:unionER} for \inst{ER} instances, and between Table~\ref{tab:SP} and Table~\ref{tab:unionSP} for \inst{SP} instances.
For example for instance classes \inst{SP\_pZero\_dRand\_$\Gamma$1} and \inst{SP\_pZero\_dRand\_Mixed} the average optimal value is equal to 247.00, hence all instances have the same optimal value for uncertainty set $\Delta$ and $\Delta^1$.
This means that the uncertainty set can be extended from $\Delta^1$ to $\Delta = \Delta^1 \cup \Delta^2$ without deteriorating the number of anchored jobs.

\begin{table}[H]
\footnotesize
\input{tab_union_ER}
\caption{\inst{ER} instances, \inst{Mixed} uncertainty}
\label{tab:unionER}
\end{table}

\begin{table}[H]
\footnotesize
\input{tab_union_SP}
\caption{\inst{SP} instances, \inst{Mixed} uncertainty}
\label{tab:unionSP}
\end{table}

\subsection{Conclusion on numerical experiments}

In the numerical experiments, we evaluated the performance of formulations for budgeted uncertainty sets, and uncertainty sets obtained by union or intersection of budgeted uncertainty sets.
For budgeted uncertainty sets, numerical tests showed that (Dom) outperforms the previously known formulation (Lay), that was dedicated to budgeted uncertainty. An advantage of formulation (Dom) over (Lay) is that the size of (Dom) is independent of the budget $\Gamma$, while (Lay) has $O(n \Gamma)$ variables.

We then investigated the impact of the parameters on the performance of (Dom). 
The influence of uniform deviation and small uncertainty budget is important, while the impact of critical precedence graphs is not significant on the efficiency of the formulation. 
Interestingly (Dom) is efficient for instances that are not matching the polyhedral characterization case, but where deviation $\dhat$ is uniform and $\Gamma$ is small. 

The proposed approach is to precompute the $\LDij$ values, then solve the obtained MIP formulation. This allowed us to solve the problem for a variety of uncertainty sets for which no linear formulation was previously investigated.
The precomputing time remains small (at most 5 seconds for partition-budgeted uncertainty sets) on the considered instances, and the MIP computation time for (Dom) is comparable to that under budgeted uncertainty.
While it was expected that (Dom) would outperform standard linearization (Std), the computational interest of applying the dominance is highlighted by the number of solved instances: on a total number of 400 instances, 358 are solved by (Dom), only 191 by (Std).

\section{Conclusion}

In the present work we investigated a versatile mixed-integer programming approach for the (AnchRob) problem. This led to a linear formulation that is applicable to any uncertainty set, provided that an algorithm for precomputing the worst-case longest paths values is available. This widens the range of uncertainty sets for which MIP formulations for (AnchRob) are known.
The keypoint for establishing a strong MIP formulation is the analysis of the combinatorial properties of (AnchRob), among which a dominance property.
This property allows for a characterization of the anchored sets polytope in interesting special cases.
The theoretical positive results for the dominance-based formulation also go together with good numerical performances around the polyhedral characterization case, for both budgeted uncertainty and uncertainty sets with several budgets.

An interesting research direction is to benefit from the obtained polyhedral results to solve large-scale instances. Another direction is to extend these results and tackle project scheduling under resource constraints. Resource constraints are indeed very salient in applications. It is worth investigating how anchored jobs may interfere with resource constraints, and see whether dominance may help again in the design of efficient linear formulations.
Another perspective is to identify other problems where an anchor-robust counterpart could be defined, and solved efficiently with mixed-integer programming. 

\appendix
\section{Series-parallel precedence graphs}
\label{app:sp}

Series-parallel digraphs are defined recursively as follows.
A digraph is \emph{series-parallel} with terminals $s$ and $t$ 
if one of the three assertions is satisfied:
\begin{itemize}
\item Its vertex-set is $\{s,t\}$ and its arc-set is $\{ (s,t) \}$;
\item (Series composition.) It is formed with two series-parallel digraphs $G_1$ and $G_2$, where terminals $t_1$ and $s_2$ have been identified;
\item (Parallel composition.) It is formed with two series-parallel digraphs $G_1$ and $G_2$, where the two pairs of terminals $s_1$ and $s_2$, and $t_1$ and $t_2$, have been identified. We assume that $G_1$ and $G_2$ have strictly more than $2$ vertices.
\end{itemize}
Series-parallel precedence graphs are series-parallel digraphs with terminals the dummy jobs $s$ and $t$.

\begin{proposition}

If $G$ is series-parallel and $G(p)$ is quasi-critical, then $G(p)$ is critical.
\end{proposition}

\begin{proof}
If $G$ is a path, then $G(p)$ is critical. If $G$ is obtained by series composition of $G^1$ and $G^2$; then any \stpath path is formed by an \odpath{s^1}{t^1} path in $G^1$ and an \odpath{s^2}{t^2} path in $G^2$. Hence if $G^1$ and $G^2$ are critical, it follows that $G$ is critical. If $G$ is obtained by parallel composition of $G^1$ and $G^2$, both critical, let $i \in G^1$ and $j \in G^2$. Then every \stpath path going through $G^1$ (resp. $G^2$) has length $L_{G^1}(s,t) = L_{G^1}(s,i) + L_{G^1}(i,t)$ (resp. $L_{G^2}(s,t) = L_{G^2}(s,j) + L_{G^2}(j,t)$). If $G$ is quasi-critical, $L_G(s,i) + L_G(i,t) = L_G(s,j)+L_G(j,t)$, and it follows that all \stpath paths of $G$ have same length.  \hfill $\square$
\end{proof}

%% file: precGraph.tex
\begin{tikzpicture}[scale=1.85, >=latex]

\tikzset{edgelabel/.style={fill=white, inner sep = 3pt, midway}}

% SOMMETS
\node[circle,draw] (s) at (-0.9,  0) {$s$};
\node[circle,draw] (1) at (0,   0.7) {1};
\node[circle,draw] (2) at (0.5*1.4,  -0.7) {2};
\node[circle,draw] (3) at (1*1.4,   0.7) {3};
\node[circle,draw] (4) at (1.5*1.4,-0.7) {4};
\node[circle,draw] (5) at (2*1.4,   0.7) {5};
\node[circle,draw] (t) at (1+2*1.4,   0) {$t$};

% ARCS
% \draw[->] (0) -- (1);
% \node[] () at (-0.45, 0.1) {\small{0}};%\small{0+0}};
% \draw[->] (6) to[bend right = 90] (1);
% \node[] () at (1.5, 1.2) {\small{-8}};%\small{1+1.5}};
\draw[->] (s) -- (1)node[edgelabel] {0};
%\node[] () at (-0.9+0.2, 0+0.4) {\small{0}};
\draw[->] (s) -- (2)node[edgelabel] {0};
%\node[] () at (-0.9+0.2, -0.4) {\small{0}};
\draw[->] (1) -- (3)node[edgelabel] {1+0.5};
%\node[] () at (0+0.3, 0.7+0.2) {\small{0}};
\draw[->] (3) -- (5)node[edgelabel] {1+0.5};
%\node[] () at (1+0.3, 0.7+0.2) {\small{0}};
\draw[->] (2) -- (4)node[edgelabel] {1+1};
%\node[] () at (0.5+0.3, -0.7+0.2) {\small{0}};
\draw[->] (3) -- (4)node[edgelabel] {1+0.5};
%\node[] () at (1+0.1, 0.7-0.4) {\small{0}};
\draw[->] (5) -- (t)node[edgelabel] {2+0.5};
%\node[] () at (2+0.3, 0.7) {\small{0}};
\draw[->] (4) -- (t)node[edgelabel] {1+0.5};
%\node[] () at (1.5+0.3, -0.7+0.2) {\small{0}};

\end{tikzpicture}

%% file: scheduleBox.tex
\begin{tikzpicture}[>=latex, xscale=1.2]

\tikzstyle{jobstyle}=[thick]
\tikzstyle{delaystyle}=[thick,  pattern=north east lines, dotted]
\tikzstyle{anchored}=[fill=black!40]
\tikzstyle{nonanchored}=[fill=black!10]

% axe
\draw[->] (-0.2, -2.5) -- (5.5, -2.5);
% gradations regulieres
\foreach \grad in {0, 1, ..., 5} 
{ \draw[dotted] (\grad ,-2.5) -- (\grad, 0.5) ;}
% ligne t=0
\draw[] (0,-2.5) -- (0, 0.5);
\node[] (zero) at (0, -2.7) {0};

% deadline
\draw[thick] (4.5,-2.5) -- (4.5, 0.5);
\node[] (M) at (4.5, -2.7) {$M$};

% job 1
\draw[anchored, jobstyle] (0,0.1) rectangle (1,-0.2);
\node[] (label1) at (0.2,-0.05) {\small{1}};
\draw[anchored, delaystyle](1,0.1) rectangle (1+0.5,-0.2);
% job 3
\draw[nonanchored, jobstyle] (1,-0.4) rectangle (2,-0.7);
\node[] (label1) at (1+0.2,-0.55) {\small{3}};
\draw[nonanchored, delaystyle](2,-0.4) rectangle (2+0.5,-0.7);
% job 5
\draw[nonanchored, jobstyle] (2.5,-0.9) rectangle (4.5,-1.2);
\node[] (label1) at (2.5+0.2,-1.05) {\small{5}};
\draw[nonanchored, delaystyle](4.5,-0.9) rectangle (4.5+0.5,-1.2);
% job 2
\draw[anchored, jobstyle] (1,-1.4) rectangle (2,-1.7);
\node[] (label1) at (1+0.2,-1.55) {\small{2}};
\draw[anchored, delaystyle](2,-1.4) rectangle (2+1,-1.7);
% job 4
\draw[anchored, jobstyle] (3,-1.9) rectangle (4,-2.2);
\node[] (label1) at (3+0.2,-2.05) {\small{4}};
\draw[anchored, delaystyle](4,-1.9) rectangle (4+0.5,-2.2);

\end{tikzpicture}

%% file: scheduleBudgeted.tex
\begin{tikzpicture}[>=latex, xscale=1.2]

\tikzstyle{jobstyle}=[thick]
\tikzstyle{delaystyle}=[thick,  pattern=north east lines, dotted]
\tikzstyle{anchored}=[fill=black!40]
\tikzstyle{nonanchored}=[fill=black!10]

% axe
\draw[->] (-0.2, -2.5) -- (5.5, -2.5);
% gradations regulieres
\foreach \grad in {0, 1, ..., 5} 
{ \draw[dotted] (\grad ,-2.5) -- (\grad, 0.5) ;}
% ligne t=0
\draw[] (0,-2.5) -- (0, 0.5);
\node[] (zero) at (0, -2.7) {0};

% deadline
\draw[thick] (4.5,-2.5) -- (4.5, 0.5);
\node[] (M) at (4.5, -2.7) {$M$};

% job 1
\draw[anchored, jobstyle] (0,0.1) rectangle (1,-0.2);
\node[] (label1) at (0.2,-0.05) {\small{1}};
\draw[anchored, delaystyle](1,0.1) rectangle (1+0.5,-0.2);
% job 3
\draw[nonanchored, jobstyle] (1,-0.4) rectangle (2,-0.7);
\node[] (label1) at (1+0.2,-0.55) {\small{3}};
\draw[nonanchored, delaystyle](2,-0.4) rectangle (2+0.5,-0.7);
% job 5
\draw[anchored, jobstyle] (2.5,-0.9) rectangle (4.5,-1.2);
\node[] (label1) at (2.5+0.2,-1.05) {\small{5}};
\draw[anchored, delaystyle](4.5,-0.9) rectangle (4.5+0.5,-1.2);
% job 2
\draw[anchored, jobstyle] (1,-1.4) rectangle (2,-1.7);
\node[] (label1) at (1+0.2,-1.55) {\small{2}};
\draw[anchored, delaystyle](2,-1.4) rectangle (2+1,-1.7);
% job 4
\draw[anchored, jobstyle] (3,-1.9) rectangle (4,-2.2);
\node[] (label1) at (3+0.2,-2.05) {\small{4}};
\draw[anchored, delaystyle](4,-1.9) rectangle (4+0.5,-2.2);

\end{tikzpicture}

%% file: exlayered.tex
\begin{tikzpicture}[yscale=1, xscale=1]
    
        \tikzset{every node/.style={inner sep=0pt, minimum height=0.7cm}}
        \tikzstyle{job}=[draw, circle, thick]
        \tikzset{arc/.style={thick,->,>=stealth}}
        \tikzset{longpath/.style={arc, very thick, blue}}
        
%        \tikzstyle{job2}=[draw, circle, thick]
        \tikzstyle{job1}=[job, shift={(0, -0)}]
        \tikzstyle{job0}=[job, shift={(0, -3)}]
        \tikzset{arc/.style={thick,->,>=stealth}}
        
        \tikzstyle{job}=[draw, circle, thick]
        \tikzset{arc/.style={thick,->,>=stealth}}
        \tikzset{every node/.style={inner sep=0pt, minimum height=0.7cm}}
        
        \tikzset{edgelabel/.style={fill=white, inner sep = 3pt, midway}}

        % dessin des sommets
        \foreach \layer in {0, 1}
            { 
            \foreach \jobname/\abs in { s/0, 1/2, 2/4, 3/6, t/8 }
            \node[job\layer] (\jobname-\layer) at (\abs, 0) {$\jobname^{\layer}$}; 

            }

        % arcs horizontaux
        \foreach \layer in {0, 1}
            {
            \draw[arc] (s-\layer) -- (1-\layer)node[edgelabel] {0};
            \draw[arc] (1-\layer) -- (2-\layer)node[edgelabel] {1};
            \draw[arc] (2-\layer) -- (3-\layer)node[edgelabel] {1};
            \draw[arc] (3-\layer) -- (t-\layer)node[edgelabel] {1};
            }
            
        % arcs transversaux
        \foreach \layer/\laybas in {1/0}
            {
            % \draw[arc] (s-\layer) -- (1-\laybas)node[edgelabel, sloped] {0};
            \draw[arc] (1-\layer) -- (2-\laybas)node[edgelabel, sloped,pos = 0.7] {1+1};
            \draw[arc] (2-\layer) -- (3-\laybas)node[edgelabel, sloped,pos = 0.7] {1+1};
            \draw[arc] (3-\layer) -- (t-\laybas)node[edgelabel, sloped, pos=0.7] {1+1};
            }
            
        %arcs montants
            
        %\foreach \layer in {0}
        %    {
        %    \draw[arc] (1-\layer) edge [bend left=20] (1-1) node[midway,right]{};
        %    \draw[arc] (5-\layer) edge [bend left=20] (5-1)node[midway,right] {};
        %    \draw[arc] (4-\layer) edge [bend left=20] (4-1)node[midway,right] {};
        %    }
            
        \draw[arc] (1-0) edge [bend left=20]node[edgelabel,midway, pos = 0.7]{0} (1-1);
        \draw[arc] (2-0) edge [bend left=20]node[edgelabel,midway,pos = 0.7]{-1} (2-1);
        \draw[arc] (3-0) edge [bend left=20]node[edgelabel,midway,pos = 0.7]{-2(1-$\frac{1}{2}$)} (3-1);

        \foreach \layer in {1}
            { 
            \foreach \jobname/\abs/\sched in { s/0/0, 1/2/0, 2/4/1, 3/6/2, t/8/3}
            \node[] (x-\jobname-\layer) at (\abs, 0+0.6) {[\sched]}; 
            }    
        \foreach \layer in {0}
            { 
            \foreach \jobname/\abs/\sched in { s/0/0, 1/2/0, 2/4/2, 3/6/3, t/8/4}
            \node[] (x-\jobname-\layer) at (\abs, -3.6) {[\sched]}; 
            }    
            
    \end{tikzpicture}
    

%% file: tab_classical_ER.tex
\hspace*{-1.4cm}
\begin{tabular}{l c c c l r c r r r r}

            instance & crit. & unif. & \inst{$\Gamma$1} & opt & & \#solved & gap & time(s) & LPGap & CPXGap \\
 \hline

%G1\_ERC\_dhatUnif\_pZero\_M50\_n300
\inst{ER\_pZero\_dUnif\_$\Gamma$1} & \checkmark & \checkmark & \checkmark & 271.90  &  (Lay) & 10 &  - &1 &  9.20\%  & % 1.40
 0.90\% \\
                            &            &            &            &         &  (Std) & 1 &  1.19\% & 144 &  5.14\%  & % 3.07
 3.07\% \\
                            &            &            &            &         &  (Dom) & 10 &  - &$<$1 & 0\%  & %Inf.00
0\% \\
\hline
%G1\_ERC\_dhatUnif\_pStd\_M50\_n300
\inst{ER\_pQCri\_dUnif\_$\Gamma$1} &            & \checkmark & \checkmark & 274.30  &  (Lay) & 10 &  - &26 &  8.39\%  & % 1.07
 0.87\% \\
                            &            &            &            &         &  (Std) & 0 &  2.05\% & - &  8.04\%  & % 4.29
 4.29\% \\
                            &            &            &            &         &  (Dom) & 10 &  - &1 &  0.42\%  & %Inf.00
 0.24\% \\
\hline
%G1\_ER\_dhatUnif\_pStd\_M50\_n300
\inst{ER\_pRand\_dUnif\_$\Gamma$1} &            & \checkmark & \checkmark & 290.40  &  (Lay) & 10 &  - &$<$1 &  3.06\%  & % 1.14
 0.06\% \\
                            &            &            &            &         &  (Std) & 10 &  - &30 &  3.02\%  & % 1.62
 1.62\% \\
                            &            &            &            &         &  (Dom) & 10 &  - &$<$1 &  0.10\%  & %Inf.00
0\% \\
\hline
%G1\_ERC\_dhatStd\_pZero\_M50\_n300
\inst{ER\_pZero\_dRand\_$\Gamma$1} & \checkmark &            & \checkmark & 214.77  &  (Lay) & 9 &  2.84\% & 21 & 34.15\%  & % 7.93
 7.93\% \\
                            &            &            &            &         &  (Std) & 9 &  4.92\% & 65 & 28.53\%  & % 4.68
 4.68\% \\
                            &            &            &            &         &  (Dom) & 9 &  1.81\% & 11 & 21.13\%  & % 3.32
 3.28\% \\
\hline
%G1\_ERC\_dhatStd\_pStd\_M50\_n300
\inst{ER\_pQCri\_dRand\_$\Gamma$1} &            &            & \checkmark & 225.80  &  (Lay) & 8 &  0.98\% & 21 & 28.54\%  & % 7.70
 7.70\% \\
                            &            &            &            &         &  (Std) & 6 &  2.30\% & 41 & 30.25\%  & % 5.41
 5.41\% \\
                            &            &            &            &         &  (Dom) & 10 &  - &31 & 18.09\%  & % 3.69
 3.69\% \\
\hline
%G1\_ER\_dhatStd\_pStd\_M50\_n300
\inst{ER\_pRand\_dRand\_$\Gamma$1} &            &            & \checkmark & 290.20  &  (Lay) & 10 &  - &$<$1 &  3.00\%  & % 0.64
 0.36\% \\
                            &            &            &            &         &  (Std) & 10 &  - &$<$1 &  3.02\%  & % 1.74
 1.74\% \\
                            &            &            &            &         &  (Dom) & 10 &  - &$<$1 &  0.82\%  & % 0.53
 0.07\% \\
\hline
%G2\_ERC\_dhatUnif\_pZero\_M50\_n300
\inst{ER\_pZero\_dUnif\_$\Gamma$2} & \checkmark & \checkmark &            & 255.50  &  (Lay) & 2 &  1.88\% & 113 & 15.11\%  & % 9.93
 9.93\% \\
                            &            &            &            &         &  (Std) & 1 &  4.20\% & 237 & 10.93\%  & % 7.58
 7.58\% \\
                            &            &            &            &         &  (Dom) & 10 &  - &53 &  3.91\%  & % 2.46
 2.46\% \\
\hline
%G3\_ERC\_dhatUnif\_pZero\_M50\_n300
\inst{ER\_pZero\_dUnif\_$\Gamma$3} & \checkmark & \checkmark &            & 243.50  &  (Lay) & 1 &  4.41\% & 111 & 19.77\%  & %13.69
13.69\% \\
                            &            &            &            &         &  (Std) & 3 &  5.48\% & 153 & 15.76\%  & % 7.78
 7.78\% \\
                            &            &            &            &         &  (Dom) & 10 &  - &105 &  7.18\%  & % 3.55
 3.51\% \\
\hline
 \end{tabular}

%% file: tab_classical_SP.tex
\hspace*{-1.4cm}
\begin{tabular}{l c c c l r c r r r r}

            instance & crit. & unif. & \inst{$\Gamma$1} & opt & & \#solved & gap & time(s) & LPGap & CPXGap \\
 \hline

%G1\_SPC\_dhatUnif\_pZero\_M50\_n300
\inst{SP\_pZero\_dUnif\_$\Gamma$1} & \checkmark & \checkmark & \checkmark & 255.90  &  (Lay) & 10 &  - &$<$1 & 17.09\%  & % 7.68
0\% \\
                            &            &            &            &         &  (Std) & 0 &  3.34\% & - &  8.64\%  & % 6.81
 6.81\% \\
                            &            &            &            &         &  (Dom) & 10 &  - &$<$1 & 0\%  & %Inf.00
0\% \\
\hline
%G1\_SPC\_dhatUnif\_pStd\_M50\_n300
\inst{SP\_pQCri\_dUnif\_$\Gamma$1} & \checkmark & \checkmark & \checkmark & 255.90  &  (Lay) & 10 &  - &$<$1 & 17.09\%  & % 4.09
0\% \\
                            &            &            &            &         &  (Std) & 0 &  6.88\% & - & 14.82\%  & % 8.96
 8.96\% \\
                            &            &            &            &         &  (Dom) & 10 &  - &$<$1 & 0\%  & %Inf.00
0\% \\
\hline
%G1\_SP\_dhatUnif\_pStd\_M50\_n300
\inst{SP\_pRand\_dUnif\_$\Gamma$1} &            & \checkmark & \checkmark & 262.80  &  (Lay) & 10 &  - &$<$1 & 14.19\%  & %16.65
 0.42\% \\
                            &            &            &            &         &  (Std) & 0 &  6.86\% & - & 13.36\%  & % 7.29
 7.29\% \\
                            &            &            &            &         &  (Dom) & 10 &  - &$<$1 &  0.05\%  & %Inf.00
0\% \\
\hline
%G1\_SPC\_dhatStd\_pZero\_M50\_n300
\inst{SP\_pZero\_dRand\_$\Gamma$1} & \checkmark &            & \checkmark & 247.00  &  (Lay) & 4 &  3.72\% & 3 & 32.46\%  & %20.95
16.43\% \\
                            &            &            &            &         &  (Std) & 2 & 18.60\% & 11 & 25.97\%  & %Inf.00
20.07\% \\
                            &            &            &            &         &  (Dom) & 4 &  9.12\% & 15 & 16.42\%  & %Inf.00
 9.70\% \\
\hline
%G1\_SPC\_dhatStd\_pStd\_M50\_n300
\inst{SP\_pQCri\_dRand\_$\Gamma$1} & \checkmark &            & \checkmark & 247.00  &  (Lay) & 4 &  3.52\% & 3 & 32.38\%  & %21.24
16.72\% \\
                            &            &            &            &         &  (Std) & 2 & 23.58\% & 30 & 35.21\%  & %54.21
27.94\% \\
                            &            &            &            &         &  (Dom) & 4 &  9.83\% & 14 & 16.54\%  & %Inf.00
 9.90\% \\
\hline
%G1\_SP\_dhatStd\_pStd\_M50\_n300
\inst{SP\_pRand\_dRand\_$\Gamma$1} &            &            & \checkmark & 268.60  &  (Lay) & 10 &  - &3 & 11.51\%  & % 6.97
 6.97\% \\
                            &            &            &            &         &  (Std) & 1 &  6.09\% & 193 & 11.05\%  & % 9.34
 9.34\% \\
                            &            &            &            &         &  (Dom) & 10 &  - &16 &  2.26\%  & % 0.92
 0.92\% \\
\hline
%G2\_SPC\_dhatUnif\_pStd\_M50\_n300
\inst{SP\_pQCri\_dUnif\_$\Gamma$2} & \checkmark & \checkmark &            & 246.66  &  (Lay) & 7 &  0.85\% & 45 & 22.65\%  & %10.58
10.58\% \\
                            &            &            &            &         &  (Std) & 0 & 11.80\% & - & 20.92\%  & %16.63
16.63\% \\
                            &            &            &            &         &  (Dom) & 9 &  0.77\% & 31 &  2.44\%  & % 1.39
 1.35\% \\
\hline
%G3\_SPC\_dhatUnif\_pStd\_M50\_n300
\inst{SP\_pQCri\_dUnif\_$\Gamma$3} & \checkmark & \checkmark &            & 246.14  &  (Lay) & 5 &  2.77\% & 57 & 26.40\%  & %16.38
16.38\% \\
                            &            &            &            &         &  (Std) & 1 & 17.42\% & 215 & 25.47\%  & %20.68
20.68\% \\
                            &            &            &            &         &  (Dom) & 7 &  2.19\% & 38 &  3.92\%  & % 3.29
 3.29\% \\
\hline
 \end{tabular}

%% file: tab_groups_ER.tex
\hspace*{-1.4cm}
\begin{tabular}{l c c c l r c r r r r}

            instance & crit. & unif. & \inst{$\Gamma$1} & opt & & \#solved & gap & time(s) & LPGap & CPXGap \\
 \hline

%group\_groupG[10, 1]\_ERC\_dhatUnif\_pZero\_M50\_n300
\inst{ER\_pZero\_dUnif\_Partition} & \checkmark & \checkmark &            & 286.50  &  (Lay) & - & - & - & - & - \\
                                &            &            &            &         &  (Std) & 10 &  - &$<$1 &  2.17\%  & % 0.33
 0.07\% \\
                                &            &            &            &         &  (Dom) & 10 &  - &$<$1 & 0\%  & %Inf.00
0\% \\
\hline
%group\_groupG[10, 1]\_ERC\_dhatUnif\_pStd\_M50\_n300
\inst{ER\_pQCri\_dUnif\_Partition} &            & \checkmark &            & 284.30  &  (Lay) & - & - & - & - & - \\
                                &            &            &            &         &  (Std) & 10 &  - &1 &  4.92\%  & % 0.37
 0.22\% \\
                                &            &            &            &         &  (Dom) & 10 &  - &$<$1 &  1.53\%  & %Inf.00
0\% \\
\hline
%group\_groupG[10, 1]\_ER\_dhatUnif\_pStd\_M50\_n300
\inst{ER\_pRand\_dUnif\_Partition} &            & \checkmark &            & 294.50  &  (Lay) & - & - & - & - & - \\
                                &            &            &            &         &  (Std) & 10 &  - &$<$1 &  1.72\%  & % 4.25
 0.05\% \\
                                &            &            &            &         &  (Dom) & 10 &  - &$<$1 &  0.34\%  & %Inf.00
0\% \\
\hline
%group\_groupG[10, 1]\_ERC\_dhatStd\_pZero\_M50\_n300
\inst{ER\_pZero\_dRand\_Partition} & \checkmark &            &            & 246.10  &  (Lay) & - & - & - & - & - \\
                                &            &            &            &         &  (Std) & 10 &  - &$<$1 & 22.03\%  & %Inf.00
0\% \\
                                &            &            &            &         &  (Dom) & 10 &  - &$<$1 & 20.22\%  & %Inf.00
0\% \\
\hline
%group\_groupG[10, 1]\_ERC\_dhatStd\_pStd\_M50\_n300
\inst{ER\_pQCri\_dRand\_Partition} &            &            &            & 260.70  &  (Lay) & - & - & - & - & - \\
                                &            &            &            &         &  (Std) & 10 &  - &$<$1 & 14.45\%  & %27.18
0\% \\
                                &            &            &            &         &  (Dom) & 10 &  - &$<$1 & 12.13\%  & %Inf.00
0\% \\
\hline
%group\_groupG[10, 1]\_ER\_dhatStd\_pStd\_M50\_n300
\inst{ER\_pRand\_dRand\_Partition} &            &            &            & 293.10  &  (Lay) & - & - & - & - & - \\
                                &            &            &            &         &  (Std) & 10 &  - &$<$1 &  2.14\%  & % 1.88
0\% \\
                                &            &            &            &         &  (Dom) & 10 &  - &$<$1 &  0.77\%  & %Inf.00
0\% \\
\hline
\end{tabular}

%% file: tab_groups_SP.tex
\hspace*{-1.4cm}
\begin{tabular}{l c c c l r c r r r r}

            instance & crit. & unif. & \inst{$\Gamma$1} & opt & & \#solved & gap & time(s) & LPGap & CPXGap \\
 \hline

%group\_groupG[10, 1]\_SPC\_dhatUnif\_pZero\_M50\_n300
\inst{SP\_pZero\_dUnif\_Partition} & \checkmark & \checkmark &            & 274.70  &  (Lay) & - & - & - & - & - \\
                                &            &            &            &         &  (Std) & 8 &  1.36\% & 14 &  4.48\%  & % 2.09
 2.01\% \\
                                &            &            &            &         &  (Dom) & 10 &  - &$<$1 & 0\%  & %Inf.00
0\% \\
\hline
%group\_groupG[10, 1]\_SPC\_dhatUnif\_pStd\_M50\_n300
\inst{SP\_pQCri\_dUnif\_Partition} & \checkmark & \checkmark &            & 273.44  &  (Lay) & - & - & - & - & - \\
                                &            &            &            &         &  (Std) & 4 &  7.42\% & 43 & 10.57\%  & % 8.29
 8.29\% \\
                                &            &            &            &         &  (Dom) & 9 &  3.22\% & 14 &  1.43\%  & %Inf.00
 0.84\% \\
\hline
%group\_groupG[10, 1]\_SP\_dhatUnif\_pStd\_M50\_n300
\inst{SP\_pRand\_dUnif\_Partition} &            & \checkmark &            & 284.90  &  (Lay) & - & - & - & - & - \\
                                &            &            &            &         &  (Std) & 8 &  2.12\% & 57 &  4.89\%  & % 2.79
 2.71\% \\
                                &            &            &            &         &  (Dom) & 10 &  - &$<$1 &  0.17\%  & %Inf.00
0\% \\
\hline
%group\_groupG[10, 1]\_SPC\_dhatStd\_pZero\_M50\_n300
\inst{SP\_pZero\_dRand\_Partition} & \checkmark &            &            & 246.87  &  (Lay) & - & - & - & - & - \\
                                &            &            &            &         &  (Std) & 7 & 11.40\% & 1 & 23.12\%  & %Inf.00
 4.54\% \\
                                &            &            &            &         &  (Dom) & 8 &  4.05\% & 33 & 19.43\%  & %Inf.00
 2.79\% \\
\hline
%group\_groupG[10, 1]\_SPC\_dhatStd\_pStd\_M50\_n300
\inst{SP\_pQCri\_dRand\_Partition} & \checkmark &            &            & 247.42  &  (Lay) & - & - & - & - & - \\
                                &            &            &            &         &  (Std) & 7 & 13.46\% & 1 & 27.36\%  & %210.01
 5.16\% \\
                                &            &            &            &         &  (Dom) & 7 &  3.28\% & $<$1 & 19.33\%  & %Inf.00
 2.66\% \\
\hline
%group\_groupG[10, 1]\_SP\_dhatStd\_pStd\_M50\_n300
\inst{SP\_pRand\_dRand\_Partition} &            &            &            & 279.10  &  (Lay) & - & - & - & - & - \\
                                &            &            &            &         &  (Std) & 7 &  4.89\% & 44 &  7.10\%  & % 5.10
 5.05\% \\
                                &            &            &            &         &  (Dom) & 10 &  - &1 &  1.85\%  & % 1.56
 0.41\% \\
\hline
 \end{tabular}

%% file: tab_union_ER.tex
\hspace*{-1.4cm}
\begin{tabular}{l c c c l r c r r r r}

            instance & crit. & unif. & \inst{$\Gamma$1} & opt & & \#solved & gap & time(s) & LPGap & CPXGap \\
 \hline

%multi\_G1r2Gr10\_ERC\_dhatUnif\_pZero\_M50\_n300
\inst{ER\_pZero\_dUnif\_Mixed} & \checkmark & \checkmark &            & 271.90  &  (Lay) & - & - & - & - & - \\
                               &            &            &            &         &  (Std) & 1 &  1.17\% & 286 &  5.14\%  & % 3.06
 3.06\% \\
                               &            &            &            &         &  (Dom) & 10 &  - &$<$1 & 0\%  & %Inf.00
0\% \\
\hline
%multi\_G1r2Gr10\_ERC\_dhatUnif\_pStd\_M50\_n300
\inst{ER\_pQCri\_dUnif\_Mixed} &            & \checkmark &            & 274.20  &  (Lay) & - & - & - & - & - \\
                               &            &            &            &         &  (Std) & 0 &  2.13\% & - &  8.08\%  & % 4.28
 4.28\% \\
                               &            &            &            &         &  (Dom) & 10 &  - &1 &  0.45\%  & %Inf.00
 0.24\% \\
\hline
%multi\_G1r2Gr10\_ER\_dhatUnif\_pStd\_M50\_n300
\inst{ER\_pRand\_dUnif\_Mixed} &            & \checkmark &            & 290.40  &  (Lay) & - & - & - & - & - \\
                               &            &            &            &         &  (Std) & 10 &  - &26 &  3.02\%  & % 1.62
 1.62\% \\
                               &            &            &            &         &  (Dom) & 10 &  - &$<$1 &  0.10\%  & %Inf.00
 0.03\% \\
\hline
%multi\_G1r2Gr10\_ERC\_dhatStd\_pZero\_M50\_n300
\inst{ER\_pZero\_dRand\_Mixed} & \checkmark &            &            & 214.77  &  (Lay) & - & - & - & - & - \\
                               &            &            &            &         &  (Std) & 9 &  5.35\% & 66 & 28.53\%  & % 4.61
 4.61\% \\
                               &            &            &            &         &  (Dom) & 9 &  1.83\% & 13 & 21.13\%  & % 3.26
 3.22\% \\
\hline
%multi\_G1r2Gr10\_ERC\_dhatStd\_pStd\_M50\_n300
\inst{ER\_pQCri\_dRand\_Mixed} &            &            &            & 225.80  &  (Lay) & - & - & - & - & - \\
                               &            &            &            &         &  (Std) & 7 &  2.61\% & 77 & 30.19\%  & % 5.09
 5.09\% \\
                               &            &            &            &         &  (Dom) & 10 &  - &22 & 18.09\%  & % 3.49
 3.49\% \\
\hline
%multi\_G1r2Gr10\_ER\_dhatStd\_pStd\_M50\_n300
\inst{ER\_pRand\_dRand\_Mixed} &            &            &            & 290.20  &  (Lay) & - & - & - & - & - \\
                               &            &            &            &         &  (Std) & 10 &  - &$<$1 &  3.02\%  & % 1.78
 1.78\% \\
                               &            &            &            &         &  (Dom) & 10 &  - &$<$1 &  0.82\%  & % 0.50
 0.07\% \\
\hline
 \end{tabular}

%% file: tab_union_SP.tex
\hspace*{-1.4cm}
\begin{tabular}{l c c c l r c r r r r}

            instance & crit. & unif. & \inst{$\Gamma$1} & opt & & \#solved & gap & time(s) & LPGap & CPXGap \\
 \hline

%multi\_G1r2Gr10\_SPC\_dhatUnif\_pZero\_M50\_n300
\inst{SP\_pZero\_dUnif\_Mixed} & \checkmark & \checkmark &            & 255.70  &  (Lay) & - & - & - & - & - \\
                               &            &            &            &         &  (Std) & 0 &  3.60\% & - &  8.72\%  & % 6.84
 6.84\% \\
                               &            &            &            &         &  (Dom) & 10 &  - &1 &  0.06\%  & %Inf.00
0\% \\
\hline
%multi\_G1r2Gr10\_SPC\_dhatUnif\_pStd\_M50\_n300
\inst{SP\_pQCri\_dUnif\_Mixed} & \checkmark & \checkmark &            & 246.75  &  (Lay) & - & - & - & - & - \\
                               &            &            &            &         &  (Std) & 2 & 24.10\% & 37 & 35.77\%  & %50.74
24.48\% \\
                               &            &            &            &         &  (Dom) & 4 & 10.15\% & 13 & 16.59\%  & %Inf.00
 9.99\% \\
\hline
%multi\_G1r2Gr10\_SP\_dhatUnif\_pStd\_M50\_n300
\inst{SP\_pRand\_dUnif\_Mixed} &            & \checkmark &            & 262.80  &  (Lay) & - & - & - & - & - \\
                               &            &            &            &         &  (Std) & 0 &  6.74\% & - & 13.36\%  & % 7.64
 7.64\% \\
                               &            &            &            &         &  (Dom) & 10 &  - &1 &  0.05\%  & % 0.48
0\% \\
\hline
%multi\_G1r2Gr10\_SPC\_dhatStd\_pZero\_M50\_n300
\inst{SP\_pZero\_dRand\_Mixed} & \checkmark &            &            & 247.00  &  (Lay) & - & - & - & - & - \\
                               &            &            &            &         &  (Std) & 2 & 19.25\% & 16 & 26.63\%  & %Inf.00
19.47\% \\
                               &            &            &            &         &  (Dom) & 4 &  9.43\% & 13 & 16.40\%  & %Inf.00
10.19\% \\
\hline
%multi\_G1r2Gr10\_SPC\_dhatStd\_pStd\_M50\_n300
\inst{SP\_pQCri\_dRand\_Mixed} & \checkmark &            &            & 247.00  &  (Lay) & - & - & - & - & - \\
                               &            &            &            &         &  (Std) & 2 & 23.86\% & 30 & 35.52\%  & %51.64
25.37\% \\
                               &            &            &            &         &  (Dom) & 4 &  9.15\% & 12 & 16.26\%  & %Inf.00
10.19\% \\
\hline
%multi\_G1r2Gr10\_SP\_dhatStd\_pStd\_M50\_n300
\inst{SP\_pRand\_dRand\_Mixed} &            &            &            & 268.40  &  (Lay) & - & - & - & - & - \\
                               &            &            &            &         &  (Std) & 1 &  6.03\% & 179 & 11.00\%  & % 9.48
 9.48\% \\
                               &            &            &            &         &  (Dom) & 10 &  - &26 &  2.34\%  & % 1.10
 1.08\% \\
\hline
 \end{tabular}